\documentclass[11pt]{article}

\usepackage[utf8]{inputenc}
\usepackage[english]{babel}

\usepackage{amsmath}
\usepackage{amssymb}
\usepackage{amsthm}
\usepackage{mathtools}
\usepackage{siunitx}
\usepackage{xcolor}

\usepackage{xifthen}

\usepackage{enumitem}

\newtheorem{thm}{Theorem}[section]
\newtheorem{cor}[thm]{Corollary}
\newtheorem{lem}[thm]{Lemma}
\newtheorem{prop}[thm]{Proposition}

\theoremstyle{definition}
\newtheorem{defin}[thm]{Definition}

\theoremstyle{remark}
\newtheorem{rem}[thm]{Remark}

\numberwithin{equation}{section}

\renewcommand\emptyset{\varnothing}

\newcommand{\sineb}{$\text{Sine}_\beta$}

\newcommand\NN{{\mathbb N}}    
\newcommand\ZZ{{\mathbb Z}}
\newcommand\ZZd{\mathbb{Z}^d}
\newcommand{\ZZdo}{\mathbb{Z}^d\setminus\{0\}}

\newcommand\RR{{\mathbb R}}
\newcommand\RRd{\mathbb{R}^d}

\renewcommand{\P}{\mathbb{P}}
\newcommand{\Pb}{\P^\beta}
\newcommand{\Pbn}{\P^\beta_n}
\newcommand{\Pbrhon}{\P^{\beta,\rho}_n}
\newcommand{\Pbnj}{\P^\beta_{n_j}}
\newcommand{\Pbnk}{\P^\beta_{n_k}}
\newcommand{\Pbstar}{\P^\beta_\star}
\newcommand{\Pbtstar}{\widetilde{\P}^\beta_\star}

\newcommand{\Q}{\mathbb{Q}}

\newcommand{\Qbstar}{\Q^\beta_\star}
\newcommand{\Bin}[2]{\mathbb{B}_{#1,#2}}

\newcommand{\Psibstar}{\Psi^\beta_\star}

\newcommand{\Conf}{\mathcal{C}}

\newcommand{\B}{\mathcal{B}}
\newcommand{\F}{\mathcal{F}}
\newcommand{\I}{\mathcal{I}}

\newcommand{\Hn}{H_n}
\newcommand{\Hback}[1]{\widetilde{H}_{#1}}
\newcommand{\hbstar}{h^\beta_\star}
\newcommand{\HDbstar}{H^\beta_{\star,\Delta}}

\newcommand{\Cbstarp}{C^\beta_{\star, \Lambda_p}}
\newcommand{\Cbstar}{C^\beta_{\star}}

\newcommand{\E}[1]{E^{}_{#1}}

\newcommand{\Zbn}{Z^\beta_n}

\newcommand{\normf}{\Vert f \Vert^{}_{\infty}}

\newcommand{\per}{\text{per}}
\newcommand{\sta}{\text{sta}}
\DeclareMathOperator{\err}{Err}
\newcommand{\diag}{\text{Diag}}

\newcommand{\toren}{\tau^{n}}

\DeclareMathOperator{\Camp}{C}

\newcommand{\leb}{\lambda^d}
\newcommand{\lebn}[1]{(\lambda^d)^{\otimes #1}}

\newcommand{\set}[2]{\{{#1}_1, \dots {#1}_{#2}\}}

\newcommand{\g}[1]{g^{(#1)}}

\newcommand{\abs}[1]{\lvert #1 \rvert}
\newcommand{\bigabs}[1]{\big\lvert #1 \big\rvert}
\newcommand{\norm}[1]{\lVert #1 \rVert}

\title{Number-Rigidity  and $\beta$-Circular Riesz gas}

\author{David Dereudre $^{1}$ and Thibaut Vasseur $^{2}$\\{\normalsize{\em}}}
 
\begin{document}

\maketitle

\footnotetext[1]{\; david.dereudre@univ-lille.fr, \; Univ. Lille, CNRS, UMR 8524, Laboratoire Paul Painlev\'e, F-59000  lille, France.}
 \footnotetext[2]{\; thibaut.vasseur@univ-lille.fr,\;Univ. Lille, CNRS, UMR 8524, Laboratoire Paul Painlev\'e, F-59000  lille, France. }

\begin{abstract}

For an inverse temperature $\beta>0$, we define the $\beta$-circular Riesz gas on $\RRd$ as any microscopic thermodynamic limit of Gibbs particle systems on the torus interacting via the Riesz potential  $g(x) = \Vert x \Vert^{-s}$. We focus on the non integrable case $d-1<s<d$. Our main result ensures, for any dimension $d\ge 1$ and inverse temperature $\beta>0$, the existence of a  $\beta$-circular Riesz gas which is not number-rigid. Recall that  a point process is said number rigid if the number of points in a bounded Borel set $\Delta$ is a function of the point configuration outside $\Delta$. It is the first time that the non number-rigidity is proved for a Gibbs point process interacting via a non integrable potential. We follow a statistical physics approach based on the canonical DLR equations. It is inspired by the recent paper \cite{DLRsinebeta} where the authors prove the number-rigidity of the \sineb{} process.

\end{abstract}

{\it key words:} Gibbs point process, DLR equations, equivalence of ensembles.

\section*{Introduction}

The pairwise  Riesz potential $g(x) = \Vert x \Vert^{-s}$ for $x \in \RRd$ is abundantly studied in several domains of mathematics as statistical mechanics, potential theory, optimization, etc. The particular case $s=d-2$ corresponds to the the Coulomb pair potential coming from the electrostatic theory. The general case  $s\le d$ is particularly interesting and challenging since the potential is not integrable at infinity.

The thermodynamic limits of associated canonical Gibbs measures at inverse temperature $\beta> 0$ are the natural microscopic equilibrium states appearing in the bulk of systems with a large number of particles. Their studies are old topics in physics and mathematical physics literature \cite{gruber1979equilibrium, lieb1975thermodynamic}. Recently a general large deviation principle at the microscopic level has been established \cite{leble2017large} with the rate function equals to the entropy plus the renormalized energy times  $\beta$. The general description of  thermodynamic limits or minimizers of the rate function is mainly open and only few results have been proved. The study of the  special case "$s=0$" ($ g(x) = - \log \Vert x \Vert$ by convention), corresponding to the so-called $\log$-gas potential, is more advanced. For instance in the case $d=1$, the thermodynamic limit or the minimizer is unique and called \sineb  \cite{valko2009continuum, killip2009eigenvalue, erbar2018one}. It is associated to the eigenvalues of random matrices \cite{dumitriu2002matrix,forrester2010log} and it is determinantal for $\beta=2$. In the present paper we study some stochastic properties of thermodynamic limits with periodic boundary conditions in the case $d-1<s<d$.
 
The definition of the Riesz interaction on the torus is not obvious. An useful variant of the Coulomb energy is the periodic Jellium which can also be defined for the Riesz potential, as explained in \cite{petrache2020crystallization}. The construction of the periodic energy for general long range potential  and periodic discrete configurations has been also given in \cite{hardin2014periodic}. The equivalence between the ground state energy of such periodic versions   has been recently achieved in \cite{lewin2019floating}. We follow these ideas to define the periodic energy $\Hn$ of a configuration of $n$ points lying on a torus $\mathbb{T}_n$ of volume $n$. It is the limit of the mean energy (with neutralizing background) of copies of the configuration, when the number of copies tends to infinity. It corresponds to the natural energy of an infinite periodic crystal. Then, for an inverse temperature parameter $\beta > 0$, the finite volume canonical Gibbs measures $\Pbn$ of energy $\Hn$ is defined on $\mathbb{T}_n$ and the microscopic thermodynamic limit is obtained by letting $n$ going to infinity. Based on the compactness of entropy level-set, the sequence $(\Pbn)_{n\geq 1}$ admits accumulation points which are called $\beta$-circular Riesz gases. Unfortunately the uniqueness of accumulation points is not guaranteed in general and probably a phase transition (non uniqueness) could occur for special values of $\beta$.   

We are mainly interested in the number-rigidity of such $\beta$-circular Riesz gases. A point process is said number rigid if the number of points in a bounded Borel set $\Delta$ is a function of the point configuration outside $\Delta$. This property has been introduced in \cite{holroyd2013}  and studied for a large variety of point processes as for instance the Gaussian zeros  and Ginibre process \cite{ghosh_rigidity_2017}, perturbed lattices \cite{peres2014rigidity}, stable matchings \cite{stableMatchingLast} or Pfaffian point processes \cite{bufetov2018number}. The number rigidity of the \sineb{} process for any inverse temperature $\beta > 0$ has been proved independently in \cite{DLRsinebeta} and \cite{reda2018rigidity} with two drastically different approaches. In \cite{reda2018rigidity} the authors follow the strategy of \cite{ghosh_rigidity_2017} which consists of controlling the variance of linear statistics whereas in \cite{DLRsinebeta} the authors lean on a statistical physics approach which has inspired the present work.
Our main result claims that for  $d-1<s<d$ and any inverse temperature $\beta>0$, there exists  a $\beta$-circular Riesz gas $\Pbstar$ which is not number rigid. More precisely, in Theorem~\ref{T:point.deletion}, we show that, conditionally on the configuration outside a compact set $\Delta$, the number of points in $\Delta$ can take any value with positive probability under $\Pbstar$. Since the \sineb{} process corresponds to the Riesz interaction with $d = 1$ and $s = 0$, our result and \cite{DLRsinebeta} ensures that in dimension $d=1$ the number-rigidity appears for $s\le d-1$. We believe that this statement occurs in any dimension $d\ge 1$ but presently it is a conjecture.


An important tool used in the present paper are the so-called  canonical Dobrushin-Lanford-Ruelle (DLR) equations. They give a local conditional description of the $\beta$-circular Riesz gases as stated in Theorem \ref{T:DLR.equations}.  This formalism has a long history since the sixties where rigorous results in statistical mechanics have been developed \cite{Ruelle70}. Recently they have been used successfully for the \sineb{} process since they are the main ingredient of the proof of its number rigidity \cite{DLRsinebeta} and, coupled with optimal transport theory, to obtain a central limit theorem for the fluctuations of linear statistics \cite{leble2018clt}. Our main contribution  is to use the canonical DLR equations to prove the non number-rigidity of $\beta$-circular Riesz gases in the opposite way of \cite{DLRsinebeta} where they proved number-rigidity of the \sineb{} process. In particular, we formalise the idea to move points toward infinity (and so to remove them) and therefore we show that the number of points in a compact set conditionally to the outside is arbitrary. Moreover this allows to go further in the Gibbsian description of $\Pbstar$. We show in Theorem \ref{T:grand.canonique} that the distribution of points in a compact $\Delta$ conditionally to the outside admits  a density with respect to the Poisson point process in $\Delta$ with a standard expression $e^{-\beta H_\Delta}$ where $H_\Delta$ is the energy of points in $\Delta$ given the outside of $\Delta$. In particular we give a sense to the energy of a point inside an infinite configuration. More technically, we pass from a canonical description of $\Pbstar$ to a grand canonical one, and our theorem is then a  reminiscent of an equivalence of ensembles result, as described in \cite{georgii2006canonical}. The full DLR description, obtained in the case $d-1 < s < d$, is also expected for $s < d-1$ (in particular for the Coulomb case $s = d-2$ as mentioned in  \cite{armstrong2019local}) but so far nothing has been proved for the moment.

The plan of the paper is the following. In a first section we present the model of $\beta$-circular Riesz gas and give the results. The second part is devoted to the proofs.
\tableofcontents

\section{Model and Results}

\setcounter{subsection}{-1}

\subsection{Notations}
For $x \in \RRd$, $\Vert x \Vert$ is the usual Euclidean norm of $x$ and $\Vert x \Vert_{\infty}$ its sup-norm. We denote by $\B(\RRd)$ the space of the Borel subsets of $\RRd$ and by $\B_b(\RRd)$ the bounded Borel subsets. The Lebesgue measure on $\RRd$ is denoted $\leb{}$. For an integer $n\geq 1$, we denote by $\Lambda_n = [-\sqrt[d]{n} /2, \sqrt[d]{n} /2]$ the box of volume $n$.

The space of \emph{point configurations} $\Conf$ is the set of all locally finite subsets of $\RRd$
\[
	\Conf := \{ \gamma \subset \RRd : \forall \Delta \in \B_b(\RRd),\, \vert \gamma \cap \Delta \vert < +\infty\},
\] where we use the notation $\vert \cdot \vert$ for the cardinality of a set. We denote by $\Conf_f$ the set of all finite point configurations, and, for $\Lambda$ a Borel subset of $\RRd$, by $\Conf_\Lambda$ the subset of $\Conf$ consisting of point configurations in $\Lambda$. Configurations can also be seen as locally finite counting measures, in particular we use the convenient notation 
$\int f(x) \gamma(dx) = \sum_{x\in\gamma} f(x)$.

For $\Delta$ a bounded Borel subset of $\RRd$, we introduce the counting function $N_\Delta$ defined for $\gamma \in \Conf$ as $N_\Delta(\gamma) = \vert \gamma \cap \Delta\vert$. We equip the set $\Conf$ with the $\sigma$-algebra generated by these counting functions $\F := \sigma(N_\Delta, \Delta \in \B_b(\RRd))$. For $\Lambda$ a Borel subset of $\RRd$, we equip $\Conf_\Lambda$ with the $\sigma$-agebra $\F_\Lambda := \sigma(N_\Delta, \Delta \in \B_b(\Lambda))$. 

We study probability measures on $(\Conf, \F)$ and refer to them as \emph{point processes} in place of their distribution.
If $\P$ is a point process, its projection to the space $(\Conf_\Lambda, \F_\Lambda)$ is denoted $\P_{|\Lambda}$ or $\P_{\Lambda}$, depending on possible ambiguities. 

For $u\in\RR$ and $n\geq 1$ an integer, $\tau_u$ denotes the translation of vector $u$, and $\tau_u^n$ the translation of vector $u$ in $\Lambda_n$ viewed as a torus.  A point process $\P$ is said \emph{stationary} if for every $u \in \RRd$ the identity $\P \circ \tau_u^{-1} = \P$ holds, and we denote by $\mathcal{P}_{\sta}(\Conf)$ the set of all stationary point processes. 

If $\P$ is a stationary point process, there exists a constant $i(\P) \in \RR^+ \cup \{+\infty\}$, called intensity of $\P$, such that for every bounded Borel subset $\Delta\subset\RRd$, $E_\P(N_\Delta) = i(\P) \leb{}(\Delta)$. If $i(\P)$ is finite, we say that $\P$ has \emph{finite intensity}. More generally, if we only have the existence of a constant $\kappa>0$ such that for any Borel subset $\Delta \subset \RRd$, $	\E \P (N_\Delta) \leq \kappa \leb{}(\Delta)$, we said that $\P$ has \emph{intensity bounded} by $\kappa$.


We denote by $\Pi$ the stationary Poisson point process of intensity equal to $1$. Recall briefly that $\Pi$ is the unique point process of intensity $1$ such that two disjoint regions of space are independent under $\Pi$.

A function $f : \Conf \rightarrow \RR$ is said \emph{local} if there exists a bounded Borel set $\Delta \subset \RRd$ such that $f$ is $\F_\Delta$-mesurable (i.e. for every $\gamma \in \Conf$, $f(\gamma) = f(\gamma_\Delta)$). The function $f$ is called \emph{tame} if for every configuration $\gamma \in \Conf$, $\abs{f(\gamma)} \leq b(1 + N_\Delta(\gamma))$, for $b$ a positive constant. The \emph{local topology} is the smallest topology on $\Conf$ such that the mappings $\P \mapsto \int f d\P$ are continuous for every local tame function $f$. 

\subsection{Circular-Riesz interaction}
We are interested in continuum particle system interacting via the Riesz potential $g$ defined for $x\in \RRd$ by $g(x) = \Vert x \Vert^{-s}$, with $s \in (d-1, d)$ (and the convention $1/0 = +\infty$). The natural pairwise energy for a finite point configuration $\gamma$ is given by 
\[
	H(\gamma) = \sum_{\{x,y\}\subset \gamma} g(x-y).
\]
Problems coming from the non-integrability at infinity of the potential can be solved by introducing a natural compensation inspired from physics (Jellium model). Charged particles interact together and with an uniform background of opposite charge. It corresponds to 
\[
	\Hback n (\gamma) = \iint_{\Lambda_n^2 \setminus \diag} g(x-y) (\gamma(dx) - dx).(\gamma(dy) - dy),  
\]
for $\gamma \in \Conf_{\Lambda_n}$. The diagonal of $(\RRd)^2$ defined by $\diag = \{ (x,y) \in (\RRd)^2 : x \neq y \}$ has been removed to avoid the singularity. This new energy is no longer pairwise, depends on the size of the box where the configuration is lying and does not have translation invariance property. We correct these difficulties by building the energy on the torus $\Lambda_n$ where the particles interact with infinitely many copies of themselves. For $\gamma$ a point configuration in $\Lambda_n$, we define  $[k] = \{r \in \ZZd : \Vert r \Vert_{\infty} \leq k\}$ and 
\[
	\gamma + n^{1/d} [k] = \{x + n^{1/d}u : x \in 	\gamma, u \in [k]\},
\]
consisting of $(2k+1)^d$ copies of $\gamma$ in order to form a point configuration in $\Lambda_{(2k+1)^d n}$. The next proposition gives a sense to the mean energy per copy when $k$ goes to infinity.

\begin{prop}\label{P:def.Hn}
For $\gamma$ a point configuration in $\Lambda_n$ such that $N_{\Lambda_n}(\gamma) = n$, we have
\[
	\lim_{k\rightarrow +\infty} \frac{\Hback {(2k+1)^dn}(\gamma + n^{1/d} [k])}{(2k+1)^d} = \sum_{\{x,y\} \subset \gamma} g_n(x-y) + n \varepsilon_n,
\]
where $\varepsilon_n$ goes to zero as $n$ goes to infinity and $g_n$ is the periodized version of the Riesz interaction defined as
\[
	g_n(x) = \sum_{k\in \ZZd} \left(
		g(x + kn^{1/d}) - \frac 1n \int_{\Lambda_n} 		g(y + kn^{1/d}) dy	
	\right).
\]
\end{prop}
We recover a pairwise interaction up to a constant which will disappear in the partition function since it does not depend on the configuration. Finally the definition of the  circular-Riesz energy on the torus $\Lambda_n$ is 
\[
	H_n(\gamma) = \sum_{\{x,y\} \subset \gamma} g_n(x-y),
\]
for any $\gamma \in \Conf_{\Lambda_n}$.

\begin{prop}\label{P:properties.Hn}The energy function $H_n$ and the periodized potential $g_n$ satisfy the following properties:
\begin{enumerate}[label=(\roman*)]
\item \label{P:properties.Hn.stable}
There exists a constant $A \leq 0$ such that for every integer $n \geq 1$ and every point configuration $\gamma\in \Conf_{\Lambda_n}$ such that $\vert \gamma \vert = n$, we have $H_n(\gamma) \geq An$.
\item \label{P:properties.Hn.inv}
For every $u \in \Lambda_n$ and every configuration $\gamma$ in $\Lambda_n$ we have $H_n(\tau_u^n(\gamma)) = H_n(\gamma)$.  
\item \label{P:properties.Hn.control}There exists a constant $c > 0$ such that for every integer $n\geq 1$ and every point $x \in \Lambda_n$ we have 
\[
	 |g_n(x) - g(x)| \leq c n^{-s/d}.
\]
\end{enumerate}
\end{prop}


Property \ref{P:properties.Hn.stable} corresponds to the so-called \emph{stability property} and is inherited from $\Hback n$. It is important to notice that it occurs only under the \emph{charge balancing condition} $\vert\gamma\vert = \leb{}(\Lambda_n)$. Property (ii) is the \emph{invariance under torus translations} of $\Hn$, which also holds for $g_n$. Finally, Property (iii) expresses that for large $n$, the periodized potential $g_n$ looks like $g$.


\subsection{From finite volume Gibbs measures to $\beta$-Circular Riesz gases}

For $\Lambda$ a bounded Borel subset of $\RRd$ such that $\leb{}(\Lambda)>0$ and $N$ an integer, let $\Bin {\Lambda}N$ be the \emph{binomial point process} on $\Conf$ which consists of throwing $N$ independent and uniformly distributed points in $\Lambda$. For an integer $n \geq 1$ and $\beta > 0$ a parameter called the \emph{inverse temperature}, we define the canonical partition function as  
\[
	Z^\beta_n = \int e^{-\beta H_n(\gamma)} \Bin {\Lambda_n}n(d\gamma).
\]

\begin{lem}\label{L:fonction.part}
For all $\beta > 0$, there exists two constants $0 < a_\beta < b_\beta$ such that for all integer $n\geq 1$
\[
	a_\beta^n \leq \Zbn \leq b_\beta^n.
\]
\end{lem}
The lemma claims that the logarithm of the partition function behaves as the volume of the system which is usual and expected for stable system in statical physics. We are now able to define the finite volume canonical Gibbs measures.

\begin{defin} For an integer $n \geq 1$ and a real number $\beta >0$, the canonical Gibbs measure in the finite volume $\Lambda_n$ with energy $\Hn$ and inverse temperature $\beta$ is the probability measure on $\Conf_{\Lambda_n}$ defined by
\[
	\Pbn = \frac 1\Zbn e^{-\beta \Hn} \Bin {\Lambda_n} n.
\]
\end{defin}
The term canonical, coming from statistical physics, refers to the fact that for a finite system, the number of particles is fixed. The next step involves the thermodynamic limit when $n$ goes to infinity. 
\begin{prop}\label{P:accumulation.points}The set $\{\Pbn\}_{n\geq 1}$ is sequentially compact for the local topology. Any accumulation point is a stationary probability measure on $\Conf$ with intensity equal to 1, called \emph{$\beta$-Circular Riesz gas}.  
\end{prop}

We do not know if the convergence occurs for the whole sequence $(\Pbn)_{n\geq 1}$. The existence of different accumulation points would correspond to a phase transition phenomenon ensuring the coexistence of different equilibrium states at the microscopic level. This phenomenon is expected for $d\ge2$ and some special values of $\beta$.  In particular, the marginal laws of $\beta$-Circular Riesz gases could be different and they are mainly intractable. Nevertheless the conditional marginal laws are explicit via the DLR equations.
\begin{rem}[$\beta$-Circular Riesz gas with arbitrary intensity]  To construct  point processes with a given intensity $\rho>0$, the finite volume Gibbs point process $\Pbrhon$ has to be built from the reference measure $\Bin {\rho^{-1/d} \Lambda_n}{n}$. According to the scaling property of $g$, the energy becomes  $\Hn^{\rho}(\gamma) = \rho^{s/d} H_n(\rho^{1/d} \gamma)$  
for $\gamma \subset \rho^{-1} \Lambda_n$ with $\vert \gamma \vert = n$.
Therefore, a $\beta$-Circular Riesz gas of intensity $\rho$ is obtained as a dilatation of factor $\rho^{-1/d}$ of a $(\beta\rho^{s/d})$-Circular Riesz gas of intensity 1. 
\end{rem}
\subsection{Canonical DLR equations}

In this section, we give a Gibbsian description of $\beta$-Circular Riesz gases via the so-called canonical DLR equations. For an integer $p \geq 1$, a bounded Borel set $\Delta\subset\RRd$ and two configurations $\eta$ and $\gamma$, we introduce the cost of moving the points of $\eta$ in $\Delta$ from the origin into the field created by $\gamma_{\Lambda_p}$  
\[
	M^{(p)}_\Delta(\eta, \gamma) = \sum_{x\in\eta_\Delta} \sum_{y\in\gamma_{\Lambda_p\setminus\Delta}} \left[ g(x-y) - g(y) \right].
\]

\begin{lem}\label{L:move.existence}
Let $\mathbb{P}$ be a point process with bounded intensity and $\Delta$ be a bounded Borel subset of $\RRd$. Then for $\mathbb{P}$-almost every point configuration $\gamma$ and for every point configuration $\eta$ the following limit exists and is finite 
\[
	M^{\vphantom{(p)}}_\Delta(\eta, \gamma) := \lim_{p\rightarrow +\infty} M^{(p)}_\Delta(\eta, \gamma). 
\]
\end{lem}

The quantity $M^{\vphantom{(p)}}_\Delta(\eta, \gamma)$ represents the cost to the points of $\eta$ in $\Delta$ from the origin into the field created by the full configuration $\gamma$.

\begin{thm}\label{T:DLR.equations} Let $\Pb$ be a $\beta$-Circular Riesz gas,  $\Delta$ be a bounded Borel subset of $\RRd$ and $f$ be a bounded measurable function. Then 
\[
	\int f(\gamma) \Pb(d\gamma) = \iint \frac 1{Z^\beta_\Delta(\gamma)}f(\eta \cup \gamma_{\Delta^c}) e^{-\beta(H(\eta) + M_\Delta(\eta, \gamma))} \Bin \Delta {N_\Delta(\gamma)} (d\eta) \Pb(d\gamma),
\]
with the normalisation constant 
\[
	Z^\beta_\Delta(\gamma) = \int e^{-\beta(H(\eta) + M_\Delta(\eta, \gamma))} \Bin \Delta {N_\Delta(\gamma)} (d\eta).
\] 
\end{thm}

The DLR equations claim that the behaviour of the points in a bounded set given an exterior configuration is gibbsian, with an energy coming from the interaction between the points and the cost of moving them from the origin to their positions in the field created by the exterior. 

Given a bounded measurable function $f$ and a bounded Borel set $\Delta$, it is useful to use the notation 
\[
	f_\Delta(\gamma) = \int \frac 1{Z^\beta_\Delta(\gamma)}f(\eta \cup \gamma_{\Delta^c}) e^{-\beta(H(\eta) + M_\Delta(\eta, \gamma))} \Bin \Delta {N_\Delta(\gamma)} (d\eta).
\]
The DLR equations can be reformulated as follow
\[
	\E \Pb (f) = \E \Pb (f_\Delta).
\]

\begin{rem}[canonical DLR equations as conditional expectations]\label{r.esperance.cond} Introducing the {$\sigma$-algebra} $\mathcal{E}_\Delta = \sigma(N_\Delta, \F_{\Delta^c})$, canonical DLR equations can be reformulated in terms of conditional expectations. Let $f$ be a bounded measurable function and $g$ be a bounded $\mathcal{E}_\Delta$-mesurable function. We have $(f\cdot g)_\Delta = f_\Delta \cdot g$ and Theorem~\ref{T:DLR.equations} gives  $\E \Pb(f\cdot g) = \E \Pb(f_\Delta \cdot g)$, which means that $f_\Delta = \E \Pb(f \mid \mathcal{E}_\Delta)$.
\end{rem}
At this point, it seems natural to ask if we can have a conditional description of the marginal law based only on $\F_{\Delta^c}$, or in other words, to ask how the number of points in $\Delta$ depends on the exterior. The number rigidity corresponds to an extreme case. 

\subsection{Non Number-Rigidity}

\begin{defin}\label{D:number.rigidity} A point process $\mathbb{P}$ is said \emph{number-rigid} if for any bounded Borel set $\Delta$, there exists a measurable function $F_\Delta$ such that for $\mathbb{P}$-almost every point configuration $\gamma$ we have $N_\Delta(\gamma) = F_\Delta(\gamma_{\Delta^c})$.
\end{defin}
In particular, for a number-rigid point process, the canonical DLR equation on the set $\Delta$ would express the conditional law with respect to $\F_{\Delta}$ in place of $\mathcal{E}_\Delta$ (see remark \ref{r.esperance.cond}). 
In the range $s\in(d-1,d)$ we exhibit a $\beta$-Circular Riesz gas which is not number-rigid. 

\begin{thm}\label{T:point.deletion}
There exists a $\beta$-Circular Riesz gas $\Pb_\star$ such that for any compact subset $\Delta\subset \RRd$  with $\leb{}(\Delta) > 0$, then for $\Pb_\star$-a.e. $\gamma$ and every integer $k$ 
\[
	\Pb_{\star}(N_\Delta = k \mid \F_{\Delta^c})(\gamma) > 0.
\]
In particular, $\Pb_\star$ is not Number-Rigid.
\end{thm}

\begin{proof}
The main result is proved Section 3.3, we only provide here the non-number rigidity. If we assume that $\Pb_\star$ is number-rigid then $\Pbstar$-almost surely, the function $N_\Delta$ is equal to a  $\F_{\Delta^c}$-mesurable function $F_\Delta$. Hence for $\Pb_\star$-almost every $\gamma$ and for every $k\geq 1$ we have
\[
	\Pb_{\star}(N_\Delta = k \mid \F_{\Delta^c})(\gamma) = E_{\Pb_\star}\left[1_{F_\Delta = k} \mid \F_{\Delta^c}\right](\gamma) = 1_{F_\Delta = k}(\gamma) =1_{N_\Delta = k}(\gamma) > 0,
\]
which is impossible.
\end{proof}

\begin{rem}[Deletion tolerance of $\Pbstar$] The previous result is related to the notion of \emph{deletion tolerance} for point processes, especially investigated in \cite{holroyd2013}. A point process $\Gamma$ (here viewed as a random variable) is said deletion tolerant if for every random variable $X$ such that $X\in\Gamma$ almost surely, the law of the new point process $\Gamma \setminus X$ is absolutely continuous with respect to the law of the point process $\Gamma$. Applying  Theorem~1 of \cite{holroyd2013} and Theorem~\ref{T:point.deletion} shows that the point process  $\Pb_\star$ is deletion tolerant. 
\end{rem}

In \cite{DLRsinebeta} they give a sense to the canonical DLR equations in the case $d=1$ and $s = d-1 = 0$ (which correspond to the logarithmic potential) and show that a point processes which satisfies the canonical DLR equations has to be number-rigid. Our result seems to suggest that all point processes (not only $\Pbstar$) satisfying the DLR equations for $d-1 < s < d$ are not number rigid. We conjecture that the number-rigidity should occur for $s \le d-1$ at any dimension $d$ (in particular for the Coulomb case $s=d-2$). But define DLR equations and exhibit a point process that satisfies them for $s \leq d-1$ and $d \geq 2$ looks challenging and has not been achieved yet. 
\subsection{Grand canonical DLR equations}\label{S:grand.canonical.DLR}
We investigate more precisely the point process $\Pbstar$ introduced in  Theorem \ref{T:point.deletion}. The next proposition provides a way to define the local energy of a point in a configuration. 
\begin{prop}\label{P:local.energy}
There exists a sequence $(\Cbstarp)_{p\geq 1}$  of measurable functions from $\NN \times \Conf$ to $\RR$ such that for $\Pbstar$-almost every configuration $\gamma$ the following limit exists and is finite  
\[
	\hbstar(x,\gamma) := \lim_{p\rightarrow +\infty} \sum_{y \in \gamma_{\Lambda_p}} g(y-x) + \Cbstarp(N_{\Lambda_p}(\gamma), \gamma_{\Lambda_p^c}). 
\]
\end{prop}
The natural candidate to define the energy of the point $x$ inside the configuration $\gamma$ is $\sum_{y\in\gamma} g(x-y)$, but it is always equal to infinity as soon as the point process is stationary with finite intensity. The sequence of functions $(\Cbstarp)_{p\geq 1}$ solve this problem by compensating the partial sums $\sum_{y\in\gamma_{\Lambda_p}} g(x-y)$ in order to obtain a limit. We call such a sequence a \emph{compensator}. 

Proposition~\ref{P:local.energy} allows to define the energy of a point in $\Pbstar$-almost every configuration. Given a bounded Borel set $\Delta$ and $\eta =\{x_1, \dots, x_n\}$ a configuration in $\Delta$, we can define the local energy of $\eta$ in $\gamma_{\Delta^c}$ as  
\[
	H^\beta_{\star,\Delta}(\eta, \gamma) = \hbstar(x_1, \gamma_{\Delta^c}) + \hbstar(x_2, \gamma_{\Delta^c}\cup\{x_1\}) + \dots + \hbstar(x_n,\gamma_{\Delta^c}\cup\{x_2, \dots x_{n-1}\}).
\]
The fact that the previous formula is well-defined is discussed in Section~\ref{S:compensator.story}. The next theorem states the grand canonical DLR equations satisfied by $\Pbstar$. 
\begin{thm}\label{T:grand.canonique}
Let $\Delta$ be a bounded Borel subset of $\RRd$ and $f$ be a bounded measurable function, we have 
\[
	\int f(\gamma) \Pbstar(d\gamma) = \iint \frac{1}{Z^\beta_{\star,\Delta}(\gamma)} f(\eta \cup \gamma_{\Delta^c}) e^{-\beta H_{\star,\Delta}^\beta(\eta, \gamma)} \Pi_\Delta(d\eta) \Pbstar(d\gamma)
\]
with the finite and nonzero normalisation constant 
\[
	Z^\beta_{\star,\Delta}(\gamma) = \int e^{-\beta H_{\star,\Delta}^\beta(\eta, \gamma)} \Pi_\Delta(d\eta).
\]
\end{thm}
As in Remark \ref{r.esperance.cond}, we can see that the previous Theorem gives a description of conditional expectations with respect to the $\sigma$-algebra $\F_{\Delta^c}$ 
\[
	E_{\Pbstar}(f\mid\F_{\Delta^c})(\gamma) =   \int \frac{1}{Z^\beta_{\star,\Delta}(\gamma)} f(\eta \cup \gamma_{\Delta^c}) e^{-\beta H_{\star,\Delta}^\beta(\eta, \gamma)} \Pi_\Delta(d\eta).
\]
More precisely, the grand canonical DLR equations claim that that the law of the point process $\Pbstar$ in any bounded Borel set $\Delta$ conditionally to the exterior has a density with respect to the Poisson point process in $\Delta$. This is the natural continuation of Theorem~\ref{T:point.deletion}, which claims that the number of points in $\Delta$ conditionally to the outside can be any integer. 
\begin{rem}[Integral compensator and fluctuations]
At this point a remaining question is the nature of the compensator introduced in Proposition~\ref{P:local.energy}. A natural candidate would be the integral compensator given by $\Cbstarp = -\int_{\Lambda_p} g(y)dy$ which is just the expectation of the partial sum. But obtain a convergence as in the proposition is difficult. It could be achieved with a good control of the  fluctuation of the number of points in bounded sets. For instance, if we can find $C >0$ and $\varepsilon >0$ such that for every integer $k\geq 1$ we have
\begin{equation}\label{E:fluctuations}
	E_{\Pbstar}(\vert N_{\Lambda_k}- k \vert) \leq C k^{s/d - \varepsilon},
\end{equation} then it is easy to show that the partial sum with the integral compensator converges. The discrepancy estimate obtained using the electrical energy in \cite{leble2017large} gives for~\eqref{E:fluctuations} an upper bound of order $k^{1/2 + s/d}$, which is not enough here. 
Such property recall the notion of \emph{hyperuniform} point process, for which $\text{Var}(N_\Lambda) = o(\leb(\Lambda))$. In particular, in dimension~1, if $\Pbstar$ is hyperuniform, then \eqref{E:fluctuations} is true as soon as $s > 1/2$. 

Unfortunately we were unable to prove such properties on the fluctuations. The nature of the compensator remains an open question although the integral compensator is the main natural candidate.  
\end{rem} 
\section{Proofs}

\subsection{Circular-Riesz interaction}
In this section we provide the proofs of properties for the periodized Riesz potential~$g_n$ and the  circular energy~$\Hn$. The property~\ref{P:properties.Hn.control} of Proposition~\ref{P:properties.Hn} can be found in Section~3.1.1. Section~3.1.2 is devoted to the proof of the Proposition~\ref{P:def.Hn} which is the main tool to prove the stability of $\Hn$ (property~\ref{P:properties.Hn.stable} of Proposition~\ref{P:properties.Hn}).

For $k$ and $j$ two integers, we introduce
\[ 
[k] = \{ r\in\ZZd: \vert r \vert \leq k\}, \quad [k\setminus j] = [k]\setminus[j] \quad \text{ and }\quad  V_k = (2k + 1)^d 
\]
with the conventions $ [0\setminus -1] = [0] $ and $V_{-1} = 0$.

\subsubsection{Properties of the periodized potential}
The following lemma contains two useful upper-bounds for $g_n$. The first property  corresponds to the third one in Proposition~\ref{P:properties.Hn}. 
\begin{lem}\label{L:majoration.n}The periodized Riesz potential $g_n$  statisfies both inequalities: 
\begin{enumerate}[label=(\roman*)]
\item 
there exists a constant $c > 0$ such that for every point $x \in \Lambda_n$
\[
	 \abs{g_n(x) - g(x)} \leq c n^{-s/d},
\]
\item 
for a given bounded Borel set $\Delta \subset \RRd$, there exists a constant $c_\Delta > 0$ such that for $n\geq 1$  large enough and for all $x\in\Delta$ and $y\in \Lambda_n$ 
\[ 
	\abs{(g_n(y-x) - g_n(y)) - (g(y-x) - g(y))} \leq c^{}_\Delta n^{-(s+1)/d}.
\]
\end{enumerate}
\end{lem}
\begin{proof}
\textit{(i)} Let $k$ be an integer such that $k \geq \sqrt{d}/2$. For $x \in \Lambda_n$ and $u\in\ZZd$, the inequality 
\(
\norm{n^{1/d}u + x}  \geq \norm{n^{1/d}u}  - n^{1/d}\sqrt{d}/2
\)
implies
\begin{align*}
	& \sum_{u\in\ZZ\setminus [k]} \left[\norm{x + n^{1/d}u}^{-s} - \frac{1}{n}\int_{\Lambda_n} \norm{y + n^{1/d}u}^{-s} dy\right] \\
	&\quad \quad \leq \sum_{u\in\ZZ\setminus[k]} \left[(\norm{n^{1/d}u} - n^{1/d}\sqrt{d} /2)^{-s} - \frac{1}{n}\int_{\Lambda_n} \norm{y + n^{1/d}u}^{-s} dy\right] \\
	&\quad \quad \leq n^{-s/d} \sum_{u\in\ZZ\setminus[k]} \left[(\norm{u} -\sqrt{d}/2)^{-s} - \int_{\Lambda_1} \norm{y + u}^{-s} dy\right]. 
\end{align*}
For $x \in \Lambda_n$ and $u\in\ZZd$, $\Vert x + n^{1/d} u \Vert \geq \Vert n^{1/d} u /2 \Vert$ and so
\begin{align*}
	\sum_{u\in[k]\setminus\{0\}} \left[\norm{x + n^{1/d}u}^{-s} - \frac{1}{n}\int_{\Lambda_n} \norm{y + n^{1/d}u}^{-s} dy\right] \leq n^{-s/d} \sum_{u\in[k]\setminus\{0\}} (\Vert u \Vert/2)^{-s}.
\end{align*}
The upper-bound follows from both previous inequalities 
\begin{align*}
	g_n(x) -g(x)  & = - \frac{1}{n}\int_{\Lambda_n} \Vert y \Vert^{-s} dy + \sum_{u\in\ZZdo} \left[\norm{x + n^{1/d}u}^{-s} - \frac{1}{n}\int_{\Lambda_n} \norm{y + n^{1/d}u}^{-s} dy\right] \\
	& \leq n^{-s/d} \Big(\sum_{u\in[k]\setminus\{0\}} (\Vert u \Vert/2)^{-s} + 
\sum_{u\in\ZZ\setminus[k]} \Big[(\norm{u} -1/2)^{-s} - \int_{\Lambda_1} \norm{y + u}^{-s} dy\Big]\Big).
\end{align*}
For the lower bound we use the inequality 
\(
\norm{x + n^{1/d} u} \leq \norm{n^{1/d}u} + n^{1/d}\sqrt{d}/2
\)
to obtain for $x\in \Lambda_n$ 
\[
g_n(x) - g(x) \geq n^{-s/d}\Big(-\int_{\Lambda_1} \Vert y \Vert^{-s} dy +  \sum_{u\in\ZZdo} \Big[(\norm{u}+  \sqrt{d} /2)^{-s} - \frac{1}{n}\int_{\Lambda_n} \norm{y + u}^{-s} dy\Big]\Big).
\]

\textit{(ii)} By definition of $g_n$
\[
	(g_n(y-x) - g_n(y)) - (g(y-x) - g(y)) = \sum_{u\in\ZZdo}\left[ \norm{y-x + un^{1/d}}^{-s} - \norm{y + un^{1/d}}^{-s}\right].
\]
Using the gradient of the potential $g$, 
\[
	\bigabs{\norm{y-x + un^{1/d}}^{-s} - \norm{y + un^{1/d}}^{-s}} \leq s \norm{x} \int_0^1 \norm{y -tx + un^{1/d}}^{-(s+1)} dt.
\]
We introduce the quantity $\rho_\Delta := \sup_{x\in\Delta} \Vert x \Vert$. For $x \in \Delta$, $y \in \Lambda_n$ and $u\in\ZZdo$, $\Vert x \Vert \leq \rho_\Delta$ and $\Vert y + n^{1/d} u \Vert \geq \Vert u n^{1/d} \Vert /2$, and so for $n\ge 1$ such that $n^{1/d} /4 \geq \rho_\Delta$
\[
	\norm{y - tx +un^{1/d}} \geq n^{1/d} \norm{u} / 4.
\]
Therefore we obtain the upper-bound 
\[
	\bigabs{[g_n(y-x) - g_n(y)] - [g(y-x) - g(y)]} \leq s4^{s+1} \rho_\Delta \sum_{u\in\ZZdo} \norm{u}^{-(s+1)}  n^{-(s+1)/d}.
\]
\end{proof}
The next lemma is useful for the proof of Proposition~\ref{P:def.Hn} in the next section.
\begin{lem}\label{L:evil.upper.bound}
There exists a constant $C >0$ such that for every integers $0 \leq j \leq k$ and every $u \in [j \setminus j-1]$ 
\[
	\int_{\Lambda_{(2k+2)^d}\setminus\Lambda_{(2k)^d}} g(y + u) dy \leq 
	\begin{cases} 
		C((k-j)^{d+1-s} + k^{d+1-s}) & \text{ if } 0 \leq j < k, \\
		C(1 + k^{d+1-s}) & \text{ if } j = k. 
	\end{cases}
\]
\end{lem}
\begin{proof}
We introduce $u_j = (j, ..., j)$, one of the $2^d$ extremal points of $[j\setminus j-1]$. By convexity of $g$, for every $u \in [j\setminus j-1]$ we have
\[
	\int_{\Lambda_{(2k+2)^d}\setminus\Lambda_{(2k)^d}} g(y + u) dy \leq \int_{\Lambda_{(2k+2)^d}\setminus\Lambda_{(2k)^d}} g(y + u_j) dy.
\]
The domain $\Lambda_{(2k+2)^d}\setminus\Lambda_{(2k)^d}$ can be partitioned into $V_{2k+2} - V_{2k}$ disjoint cubes (which are translations of $\Lambda_1$). The distance between the point $u_j$ and the domain $\Lambda_{(2k+2)^d}\setminus\Lambda_{(2k)^d}$ is equal to $k-j$. To provide an upper-bound of the integral, we count the number of cubes which are at distance $k-i$ from $u_j$ (for the supremum norm $\Vert \cdot \Vert_\infty$) with $0 \leq i \leq j$. there exists a constant $A>0$ such that
\begin{itemize}
	\item the number of cubes at distance $k-j$ is less than $A(k-j)^{d-1}$ (or $A$ if $j = k$),
	\item the number of cubes at distance $k-i$ with $ 0 < i < j$   is less than $A(k-i)^{d-2}$,
	\item the number of cubes at distance greater than $k$ is less than $Ak^{d-1}$.
\end{itemize} 
If a cube is at distance $k -i$ from $u_j$, then its interaction with $u_j$ is less than $(k-i)^{-s}$ if $i\neq k$, and less than a constant $b > 0$ if $i =k$. Then there exists a constant $C >0$ such that for $j\neq k$
\begin{align*}
	\int_{\Lambda_{(2k+2)^d}\setminus\Lambda_{(2k)^d}} g (y + u_j) dy & \leq A k^{d-1} k^{-s} + A\sum_{i=1}^j (k-i)^{d-2} (k-i)^{-s} 
	\\
	& \leq  C((k-j)^{d-1-s} + k^{d-1-s}),
\end{align*}
and for $j = k$ 
\begin{align*}
	\int_{\Lambda_{(2k+2)^d}\setminus\Lambda_{(2k)^d}} g (y + u_j) dy & \leq A k^{d-1} k^{-s}+ A\sum_{i=1}^{k-1} (k-i)^{d-2} (k-i)^{-s} + Ab 
	\\
	& \leq C(1 + k^{d-1-s}).
\end{align*}
\end{proof}

\subsubsection{Proof of Proposition \ref{P:def.Hn}}

\begin{proof}[Proof of Proposition \ref{P:def.Hn}]
We recall the useful notations, for $k$ and $j$ two integers,
\[ 
	[k] = \{ r\in\ZZd: \vert r \vert \leq k\}, \quad [k\setminus j] = [k]\setminus[j] \quad \text{ and }\quad  V_k = (2k + 1)^d, 
\]
with the conventions $ [0\setminus -1] = [0] $ and $V_{-1} = 0$. The domain $\Lambda_{(2k+1)^d n}$ can be decomposed into cubes indexed by the set $[k]$ as  
\[
	\Lambda_{(2k+1)^d n} = \bigcup_{u\in[k]} (\Lambda_n + n^{1/d}u).
\]
The energy $\Hback {(2k+1)^d n}$ can be written as a sum of interactions between the different cubes
\begin{multline*}
	2\Hback {(2k+1)^d n}(\gamma +  n^{1/d} [k]) 
	\\ 
	= 
	\sum_{u,v\in[k]} \iint_{\Lambda_n^2 \setminus \diag} g(x-y+n^{1/d}(u-v)) \big(\gamma - \lambda^d_{\Lambda_n}\big)(dx)\big(\gamma - \lambda^d_{\Lambda_n}\big)(dy).
\end{multline*}
If we expand the terms in the integral we obtain 
\begin{multline*}
	2\Hback {(2k+1)^d n}(\gamma + n^{1/d}[k]) = 2\sum_{u\in[k]} \sum_{\{x,y\}\subset\gamma} \sum_{v\in[k]} g(x-y+n^{1/d}(u-v)) + \vert \gamma \vert \sum_{\substack{u,v\in[k]\\ u\neq v}}g(n^{1/d}(u-v))
	\\
 	- 2\sum_{x\in\gamma} \sum_{u,v \in[k]} \int_{\Lambda_n} g(x-y+n^{1/d}(u-v))dy+\sum_{u,v\in[k]}\iint_{\Lambda_n^2} g(x-y+n^{1/d}(u-v))dxdy.
\end{multline*}
Taking $k = +\infty$ in the previous equality and using a  change of variable in the integrals, the two last terms do not depend on $x$ and it leads to simplifications as soon as $\vert \gamma \vert = n$. Unfortunately, $k$ is finite, and in order to obtain the simplifications, we introduce for $x\in\Lambda_n$ the error due to translate $x$ to the center of $\Lambda_n$ 
\[
\err_k^{(n)}(x) = \sum_{u,v\in[k]} \int_{\Lambda_n} g(x-y+n^{1/d}(u-v))dy -  \sum_{u,v\in[k]} \int_{\Lambda_n} g(y+n^{1/d}(u-v))dy. 
\]
Using these errors terms and the hypothesis $\vert \gamma\vert = n$ (the charge balance, crucial here), we obtain
\begin{align*}
	& 2\Hback {(2k+1)^d n}(\gamma + n^{1/d}[k]) \\
	& \quad = 2\sum_{u\in[k]} \sum_{\{x,y\}\subset\gamma} \sum_{v\in[k]}\left( g(x-y+n^{1/d}(u-v)) -\frac{1}{n}\int_{\Lambda_n} g(z + n^{1/d}(u-v))dz\right)\\
 	& \quad \quad + n \sum_{\substack{u,v\in[k] \\ u\neq v}}\left(g(n^{1/d} (u-v))- \frac{1}{n}\int_{\Lambda_n} g(z+n^{1/d}(u-v))dz\right) + n\int_{\Lambda_n} g(z) dz \\
	&  \quad \quad - 2\sum_{x\in\gamma} \err_k^{(n)}(x) + \int_{\Lambda_n}\err_k^{(n)}(x) dx.
\end{align*}
We postpone the proof of the following lemma which gives the control of errors.
\begin{lem}\label{L:evil.lemma}
If $n$ is fixed and $k$ goes to infinity  
\[
	\sup_{x \in \Lambda_n} \bigabs{\err^{(n)}_k(x)} = o(V_k).
\]
\end{lem}
\noindent
At this point, using Lemma~\ref{L:evil.lemma},
\begin{align*}
& 2\Hback {(2k+1)^d n}(\gamma + n^{1/d}[k]) \\
& \quad = 2\sum_{u\in[k]} \sum_{\{x,y\}\subset\gamma} \sum_{v\in[k]}\left( g(x-y+n^{1/d}(u-v)) -\frac{1}{n}\int_{\Lambda_n} g(z + n^{1/d}(u-v))dz\right) \\
  & \quad \quad+ n \sum_{\substack{u,v\in[k] \\ u\neq v}}\left(g(n^{d}(u-v))- \frac{1}{n}\int_{\Lambda_n} g(z+n^{1/d}(u-v))dz\right) + o(V_k).
\end{align*}
The expected limit obtained after dividing by $V_k$ and letting $k$ goes to infinity is equal to $2\Hn(\gamma) + ng_n^*(0)$ where 
\[
	g_n^*(0) = \sum_{u\in\ZZdo} \left(g(n^{1/d}u) - \frac 1n \int_{\Lambda_n} g(z+n^{1/d}u)dz \right).
\] 
So subtracting it, that leads to
\begin{align*}
& V_k \left(2\Hn(\gamma) + n g^*_n(0)\right)  - 2\Hback{(2k+1)^d n}(\gamma + n^{1/d}[k]) 
\\
& \quad =
2\sum_{u\in[k]} \sum_{\{x,y\}\subset\gamma} \sum_{v\in\ZZd\setminus[k]}\left( g(x-y+n^{1/d}(u-v)) -\frac{1}{n}\int_{\Lambda_n} g(z + n^{1/d}(u-v))dz\right)
\\
& \quad\quad  + n \sum_{u\in[k]} \sum_{v\in\ZZd\setminus[k]}\left(g(n^{1/d}(u-v))- \frac{1}{n}\int_{\Lambda_n} g(z+n^{1/d}(u-v))dz\right) + o(V_k).
\end{align*}
The set $[k]$ can be decomposed into shells $[j\setminus j-1]$ with $j= 0, ...,  k$ and so
\begin{align*}
& V_k \left(2\Hn(\gamma) + n g^*_n(0)\right)  - 2\Hback {(2k+1)^d n}(\gamma + n^{1/d}[k]) 
\\
& \quad =
2\sum_{j=0}^k \sum_{u\in[j\setminus j-1]} \sum_{\{x,y\}\subset\gamma} \sum_{v\in\ZZd\setminus[k]}\left( g(x-y+n^{1/d}(u-v)) -\frac{1}{n}\int_{\Lambda_n} g(z + n^{1/d}(u-v))dz\right)
\\
& \quad \quad + n \sum_{j=0}^k \sum_{u\in[j\setminus j-1]} \sum_{v\in\ZZd\setminus[k]}\left(g(n^{1/d}(u-v))- \frac{1}{n}\int_{\Lambda_n} g(z+n^{1/d}(u-v))dz\right) +o(V_k).
\end{align*}
Since for $u \in [j\setminus j-1]$, the set $\{u -v, v\in\ZZd\setminus [k]\}$ is included into $\ZZd\setminus [k-j]$, we have the upper bound
\begin{align}\label{E:inequality.cesaro}
\begin{split}
	& \left|2 \Hback {(2k+1)^dn}(\gamma + n^{1/d}[k]) - V_k \left(2\Hn(\gamma) + n g^*_n(0)\right)\right| 
	\\
	& \quad \leq 
2\sum_{j=0}^k \sum_{u\in[j\setminus j-1]} \sum_{\{x,y\}\subset\gamma} \sum_{v\in\ZZd\setminus[k-j]}\left| g(x-y+n^{1/d}v)) -\frac{1}{n^d}\int_{\Lambda_n} g(z + n^{1/d}v)dz\right|
	\\
 	& \quad \quad + n\sum_{j=0}^k \sum_{u\in[j\setminus j-1]} \sum_{v\in\ZZd\setminus[k-j]}\left|g(n^{1/d}v)- \frac{1}{n}\int_{\Lambda_n} g(z+n^{1/d}v)dz\right|.
\end{split}
\end{align}
By introducing the truncation errors
\begin{multline*}
	R_j = \sum_{\{x,y\}\subset\gamma} \sum_{v\in\ZZd\setminus[j]}\left| g(x-y+n^{1/d}v)) -\frac{1}{n}\int_{\Lambda_n} g(z + n^{1/d}v)dz\right| 
	\\
	 + n \sum_{v\in\ZZd\setminus[j]} \left|g(n^{1/d}v)- \frac{1}{n}\int_{\Lambda_n} g(z+n^{1/d}v)dz\right|,
\end{multline*}
the inequality (\ref{E:inequality.cesaro}) can be rewritten as 
\[
	\left|2\Hback {(2k+1)^dn}(\gamma +  n^{1/d}[k]) - V_k \left(2\Hn(\gamma) + n g^*_n(0)\right)\right| \leq \sum_{j=0}^k (V_j - V_{j-1}) R_{k-j} + o(V_k).
\]
To conclude, it remains to see that $(1/V_k)\sum_{j=0}^k (V_{k-j} - V_{k-j-1})R_j$ goes to $0$ as $k$ goes to infinity, which is a consequence of Cesar\'o's Lemma. Then the result of Proposition~\ref{P:def.Hn} is proved with $\varepsilon_n = g_n^*(0)/2 = n^{-s/d}g_1^*(0)/2$, which goes to $0$ when $n$ goes to infinity.
\end{proof}
\begin{proof}[Proof of Lemma \ref{L:evil.lemma}]
For $x \in \Lambda_n$, we need to control the error term 
\begin{align*}
 \err^{(n)}_k(x) & = \sum_{u,v\in[k]} \int_{\Lambda_n} g(x-y+n^{1/d}(u-v))dy -  \sum_{u,v\in[k]} \int_{\Lambda_n} g(y+n^{1/d}(u-v))dy. 
\end{align*}
After changes of variables,
\begin{multline*}
\err^{(n)}_k(x) =
\sum_{u\in[k]} \int_{(x +\Lambda_{(2k+1)^dn})\setminus\Lambda_{(2k+1)^dn}} g(y + n^{1/d}u) dy \\ - \sum_{u\in[k]}  \int_{\Lambda_{(2k+1)^dn}\setminus(x + \Lambda_{(2k+1)^dn})} g(y + n^{1/d}u))dy.
\end{multline*}
From the inclusions
\[
(x +\Lambda_{(2k+1)^dn})\setminus\Lambda_{(2k+1)^dn}  \;\cup\; \Lambda_{(2k+1)^dn}\setminus(x + \Lambda_{(2k+1)^dn})\subset \Lambda_{(2k+2)^dn} \setminus \Lambda_{(2k)^dn},
\]
we obtain an upper-bounded independent of $x$  
\begin{equation*}
\bigabs{E^{(n)}_k(x)}
\leq
n^{(d-s)/d} \sum_{u\in[k]} \int_{\Lambda_{(2k+2)^d}\setminus\Lambda_{(2k)^d}} g(y + u) dy.
\end{equation*}
Decomposing $[k]$ into shells $[j\setminus j-1]$, with $j = 0, ..., k$ and applying Lemma \ref{L:evil.upper.bound} that leads to 
\[
\sup_{x\in\Lambda_n} \bigabs{E^{(n)}_k(x)} 
\leq
C n^{(d-s)/d} \sum_{j=0}^k (V_j - V_{j-1}) ((k-j)^{d-1-s} + k^{d-1-s}).
\end{equation*}
It remains to show that the right hand side is dominated by $V_k$ asymptotically. There exists a constant $c_d$ such that $V_j - V_{j-1} \leq c_d j^{d-1}$  and so
\[
 \frac{1}{V_k} \sum_{j=0}^k (V_j - V_{j-1}) (k-j)^{d-1-s} \leq c_d 2^{-d} k^{d-1-s} \frac{1}{k}\sum_{j=0}^{k-1} (j/k)^{d-1} (1 - j/k)^{d-s-1}.
\]
The right hand side of the previous inequality is a Riemann sum and therefore 
\[
	\lim_{k\to+\infty} \frac{1}{k}\sum_{j=0}^{k-1} (j/k)^{d-1} (1 - j/k)^{d-s-1} = \int_0^1 t^{d-1}(1-t)^{d-s-1}dt. 
\]
Since $s > d-1$, we conclude that 
\[
	\lim_{k\to+\infty} \frac{1}{V_k} \sum_{j=0}^k (V_j - V_{j-1}) ((k-j)^{d-1-s} + k^{d-1-s}) = 0.
\]
\end{proof}

\subsubsection{Stability of $\Hn$.}\label{S:stability}

We start by proving the stability of $\Hback n$ (a similar proof is given in Appendix~A of \cite{lewin2019floating}). The Riesz potential can be decomposed in two parts $g = \g1 + \g2$ with 
\[
	\g1(x) = \frac{1}{(1+\Vert x \Vert^2)^{s/2}} \quad \text{ and } \quad \g2(x) = \frac{1}{\Vert x \Vert^s} - \frac{1}{(1+\Vert x \Vert^2)^{s/2}}.
\]
The potential $\g2$ is non-negative and integrable, and $\g1$ is of \emph{positive type}, which means that there exists a probability measure $\nu$ on $\RRd$ such that 
\[
	\g1(x) = \int_{\RRd} e^{it\cdot x} \nu(dt).
\]
Indeed $\g1$ is the Fourier transform of the Bessel kernel given by 
\[
	G_s(y) =\frac{1}{(4\pi)^{2d} \Gamma(s/2)} \int_0^{+\infty} \frac{e^{-u-\Vert y \Vert^2}/(4u)}{u^{1+(d-s)/2}}du.
\]
So the energy $\Hback n$ is decomposed in two terms $\Hback n = \Hback n ^{(1)} + \Hback n ^{(2)}$ studied separately. The contribution of $g^{(1)}$ is lower bounded using the positive type property  
\begin{align*}
\Hback n ^{(1)}(\gamma) & = \frac{1}{2} \iint_{\Lambda_n^2 \setminus\diag} \g1(x-y)\big(\gamma - \leb{}\big)(dx)\big(\gamma - \leb{}\big)(dy) \\
& = -\frac 12 \g1(0) \vert \gamma \vert + \frac 12 \iint_{\Lambda_n^2} \g1(x-y)\big(\gamma - \leb{}\big)(dx)\big(\gamma - \leb{}\big)(dy) \\
& = -\frac 12 \vert \gamma \vert + \frac 12 \iiint_{\Lambda_n^2\times \RRd} e^{it\cdot(x-y)}\big(\gamma - \leb{}\big)(dx)\big(\gamma - \leb{}\big)(dy) \nu(dt)\\
& = -\frac 12 \vert \gamma \vert + \frac 12 \int_{\RRd} \left| \int_{\Lambda_n} e^{it\cdot x}\big(\gamma - \leb{}\big)(dx)\right|^2\nu(dt)\\
& \geq -\frac 12 \vert \gamma \vert.
\end{align*}
The contribution of $\g2$ is lower-bounded using successively its non-negativity and  integrability 
\begin{align*}
\Hback n ^{(2)}(\gamma) & = \frac 12 \iint_{\Lambda_n \setminus \diag} \g2(x-y)\big(\gamma - \leb{}\big)(dx)\big(\gamma - \leb{}\big)(dy) \\
&\geq - \sum_{x\in\gamma} \int_{\Lambda_n} \g2(x -y)  dy   \\
& \geq - \sum_{x\in\gamma} \int_{\RRd} \g2(x-y) dy \\
& \geq - \vert \gamma \vert \int_{\RRd} \g2(y) dy.
\end{align*}
In particular, for a point configuration $\gamma \in \Conf_{\Lambda_n}$,
\[
\Hback {(2k+1)^d n}(\gamma + n^{1/d}[k]) \geq -\left( \frac{1}{2} + \int_{\RRd} \g2(y) dy \right)(2k+1)^d \vert \gamma \vert,
\]
and if $\vert \gamma \vert = n$, Proposition~\ref{P:def.Hn} implies that 
\[ 
	\Hn(\gamma) +  n\varepsilon_n \geq -\left( \frac{1}{2} + \int_{\RRd} \g2(y) dy \right) n.
\] 
Since $\varepsilon_n$ goes to $0$ as $n$ goes to infinity, we can conclude that $\Hn$ satisfies the property~\ref{P:properties.Hn.stable} of Proposition~\ref{P:properties.Hn}.

\subsection{$\beta$-Circular Riesz gases}
In this section we provide the proofs related to the existence of the $\beta$-Circular Riesz gases. First we introduce a family of configurations, the perturbed lattices, which are used to obtain the lower bound of the partition function. This control is crucial to verify that the sequence $(\Pbn)_{n\geq1}$ satisfies our compactness tool described in the last section.

\subsubsection{The lattice perturbed configurations}
\newcommand{\WLambda}{\widetilde{\Lambda}}
For $n\geq 1$ we use the notation $r_n = \lfloor n^{1/d} \rfloor$ and introduce the translated boxes \(\WLambda_{r_n} = \Lambda_{r_n} - (n^{1/d} -r_n)u\) and $\WLambda_{r_n+1} = \Lambda_{r_n+1} - (n^{1/d} -r_n)u$ where $u = (1,\dots, 1)$. They satisfy $\WLambda_{r_n} \subset \Lambda_n \subset \WLambda_{r_n +1}$ and there exist disjoint cubes $(C_j)_{1 \leq j \leq (r_n+1)^d}$ (i.e. translations of $\Lambda_1$) such that 
\begin{equation}\label{E:decomposition.Lambda}
	\WLambda_{r_n} = \bigcup_{1 \leq j \leq r_n^d} C_j \quad \text{ and }\quad  \WLambda_{r_n + 1} = \bigcup_{1 \leq j \leq (r_n+1)^d} C_j.
\end{equation}
Then we define $\widetilde{C}_j = C_j \cap \Lambda_n$, which is a box whose center is denoted $u_j$, and for $\delta \in ]0,1[$, $\Delta_j = \{x \in \widetilde{C}_j : \Vert x - u_j \Vert_\infty \leq \delta/2 \}$ which is a neighbourhood of the center $u_j$. A point configuration $\gamma = \{ x_1, ..., x_n\} \subset \Lambda_n$ is a \emph{perturbed lattice} if it satisfies
\begin{itemize}
\item[-] for every $j \in \{1, \dots, r_n^d\}$, $\vert \gamma\cap \Delta_j \vert = 1$,
\item[-] for every $j \in \{1, \dots, n\}$, $\vert \gamma\cap \Delta_j \vert \in \{0,1\}$. 
\end{itemize} 
In other words, $\gamma$ has one point in the neighbourhood of each center of the cubes in $\WLambda_{r_n}$ and its other points are in the neighbourhood of centers of the boxes in $\Lambda_n \setminus \WLambda_{r_n}$. We denote $\Conf_{\delta, n}$ the set of perturbed lattices defined above. 

\begin{prop}\label{P:energy.perturbed.lattice}
There exists a constant $c > 0$ such that for every $\delta \in ]0, 1/2[$ and $n\geq 1$, if $\gamma \in \Conf_{\delta, n}$ then $|\Hn(\gamma)| \leq c n$.
\end{prop}

\begin{proof} According to Proposition~\ref{P:def.Hn}, we only have to prove that for every perturbed lattice $\gamma \in \Conf_{\delta,n}$ and  every integer $k$
\[
\Hback {(2k+1)^dn}(\gamma_k) \leq c(2k+1)^dn.
\] 
where $\gamma_k = \gamma + n^{1/d}[k]$. As in the proof of the stability property in Section \ref{S:stability}, we use the decompositions $g = g^{(1)} + g^{(2)}$ and $\Hback n = \Hback n ^{(1)} + \Hback n ^{(2)} $. The two contributions are studied separately.

\textit{Contribution of $\Hback n ^{(1)}$ :} we need to control 
\[
	\Hback {(2k+1)^dn}^{(1)}(\gamma_k) = \frac 12\iint_{\Lambda_{(2k+1)^d n} \times \Lambda_{(2k+1)^d n} } g^{(1)}(x- y) (\gamma_k(dx) -dx)(\gamma_k(dy) - dy).
\]
First we look at the interaction of a point $y \in \Lambda_{(2k+1)^dn}$ with the lattice part $\gamma_{\Lambda_{r_n}} + n^{1/d}[k]$ and his background
\[
	\int_{\Lambda_{r_n} + n^{1/d} [k]} g^{(1)}(x-y)(\gamma_k(dx) - dx) = \sum_{u\in[k]} \int_{\WLambda_{r_n}} g^{(1)}(x+n^{1/d}u-y)(\gamma(dx) - dx).
\]
According to the definition of a perturbed lattice configuration and the decomposition given in \eqref{E:decomposition.Lambda}
\[
\int_{\WLambda_{r_n}} g^{(1)}(x+n^{1/d}u-y)(\gamma(dx) - dx) = \sum_{j=1}^{r_n^d} \int_{C_j} (g^{(1)}(x_j + n^{1/d}u - y) - g^{(1)}(x+n^{1/d} u - y))dx,
\]
where $\gamma_{\WLambda_n} = \{x_1, \dots, x_{r_n^d}\}$. Again using the gradient of $g^{(1)}$, 
\[
	g^{(1)}(x_j+n^{1/d} u-y) - g^{(1)}(x+n^{1/d} u-y) = \int_0^1 \langle x_j - x, \nabla g^{(1)}\left(x+n^{1/d} u-y + t(x_j - x)\right) \rangle dt.
\]
From the upper-bound  $\Vert \nabla g^{(1)}(x) \vert \leq (1 + \Vert x \Vert^2)^{(s+1)/2}$ and the inequality
\[
\Vert x+n^{1/d} u-y + t(x_j - x) \Vert \leq \Vert u_j + n^{1/d} u - y \Vert + \sqrt{d}/2 + \delta, 
\]
we obtain that for $v_y \in \{u_j + u : j\in\{1,\dots,r_n^d\}, u\in[k]   \}$ such that $\Vert v_y - y \Vert \leq \sqrt{d}$
\begin{multline*}
	\left| \int_{\Lambda_{r_n} + n^{1/d} [k]} g^{(1)}(x-y)(\gamma_k(dx) - dx)\right|  \\ \leq \frac 12  \sqrt{d} \sum_{u\in[k]} \sum_{j=1}^{r_n^d} (1 + (\Vert u_j + n^{1/d} u - v_y \Vert + 2\sqrt{d})^2)^{(s+1)/2}.
\end{multline*}
At the end we find an upper bound independent of $y$ 
\[
	\left| \int_{\Lambda_{r_n} + n^{1/d} [k]} g^{(1)}(x-y)(\gamma_k(dx) - dx) \right| \leq  \frac 12 \sqrt{d} \leq \sum_{u\in \ZZd} (1 + (\Vert u \Vert + 2\sqrt{d})^2)^{(s+1)/2} =: C_1.
\]
It leads to
\[
	\left| \Hback {(2k+1)^dn}^{(1)}(\gamma_k)  \right| \leq 3 (2k+1)^d n C_1 + \frac 12 \iint_{B_{k,n} \times B_{k,n}} g^{(1)}(x-y) (\gamma_k(dx) -dx)(\gamma_k(dy) - dy),
\]
where $B_{k,n} = \bigcup_{u\in[k]}(\Lambda_n\setminus\WLambda_{r_n} + n^{1/d} u)$, which corresponds to the peripheral points. For $y \in \Lambda_{n}$ we look at 
\begin{equation}\label{E:second.contrib}
	\int_{B_{k,n}} g^{(1)}(x-y) (\gamma_k(dx) -dx) = \sum_{u\in[k]} \int_{\Lambda_n \setminus \WLambda_{r_n}} g^{(1)}(x + n^{1/d}u -y) (\gamma(dx) -dx).
\end{equation}
There exist Borel sets $(D_j)_{1\leq j \leq n-r_n^d}$ of volume equal to one forming a partition of $\Lambda_n \setminus \WLambda_{r_n}$ and so
\[
	 \int_{\Lambda_n \setminus \WLambda_{r_n}} g^{(1)}(x + n^{1/d}u -y) (\gamma(dx) -dx) = \sum_{j=1}^{n-r_n^d} \int_{D_j} (g^{(1)}(x_j + n^{1/d}u - y) - g^{(1)}(x + n^{1/d}u - y)) dx.
\] 
Again using the gradient of $g^{(1)}$, 
\[
	g^{(1)}(x_j+n^{1/d} u-y) - g^{(1)}(x+n^{1/d} u-y) = \int_0^1 \langle x_j - x, \nabla g^{(1)}\left(x+n^{1/d} u-y + t(x_j - x)\right) \rangle dt.
\]
Since  $\Vert x_j - x\Vert \leq n^{1/d}\sqrt{d}$ and  $y \in \Lambda_n$ we have the lower bound 
 \[
 \Vert x+n^{1/d} u-y + t(x_j - x)\Vert \geq n^{1/d}(\Vert u \Vert - 3/2), 
 \]
 and using it for each $u\in[k]\setminus[1]$ in \eqref{E:second.contrib},
\begin{multline*}
	\left| \int_{B_{k,n}} g^{(1)}(x-y) (\gamma_k(dx) -dx) \right| \leq \left|\int_{\Lambda_n \setminus \WLambda_{r_n} + n^{1/d}[1]} g^{(1)}(x + n^{1/d}u -y) (\gamma(dx) -dx)\right| \\
	+ (n - r_n^d) n^{1/d} \sqrt{d} \sum_{u\in\ZZd\setminus[1]} (1 + n^{2/d} (\Vert u \Vert - 3/2)^2)^{-(s+1)/2}.
\end{multline*}
Using that $n - r_n \leq d n^{(d-1)/d}$, the second term of the right hand side is upper-bounded 
\begin{multline*}
(n - r_n^d) n^{1/d} \sqrt{d} \sum_{u\in\ZZd\setminus[1]} (1 + n^{2/d} (\Vert u \Vert - 3/2)^2)^{-(s+1)/2} \\ \leq n^{(d-1-s)/d} d^{3/2}  \sum_{u\in\ZZd\setminus[1]} (\Vert u \Vert - 3/2)^{-(s+1)} \leq C_2,
\end{multline*}
where $C_2$ is a constant independent of $n$. Moreover, for $C_2>0$ chosen large enough, with a argument similar to the proof of Lemma \ref{L:evil.upper.bound}, we can show that for all $n\geq 1$  and $y \in \Lambda_n$ 
\[
	\left|\int_{\Lambda_n \setminus \WLambda_{r_n} + n^{1/d}[1]} g^{(1)}(x + n^{1/d}u -y) (\gamma(dx) -dx)\right| \leq C_2(1 + n^{(d-1-s)/d}) \leq 2C_2.
\]
To conclude, we have obtained 
\begin{align*}
	\left| \Hback {(2k+1)^dn}^{(1)}(\gamma_k)  \right| 	& \leq 3 (2k+1)^d n C_1 + 2(2k+1)^d(n-r_n^d) C_2 \\
	& \leq (2k+1)^d n (3C_1 + 2C_2). 
\end{align*}

\textit{Contribution of $\Hback n ^{(2)}$ : }it remains to control 

\[
	\Hback {(2k+1)^dn}^{(2)}(\gamma_k) = \frac 12\iint_{\Lambda_{(2k+1)^d n} \times \Lambda_{(2k+1)^d n} } g^{(2)}(x- y) (\gamma_k(dx) -dx)(\gamma_k(dy) - dy),
\]
which is simpler, thanks to the integrability of $g^{(2)}$. For instance, for $y \in \RRd$
\[
	\left|\int_{(2k+1)^dn} g^{(2)}(x-y) dx \right| \leq \int_{\RRd} |g^{(2)}|(x -y) dx = \int_{\RRd} |g^{(2)}|(x) dx. 
\]
Thus we obtain 
\[
	\left| \Hback {(2k+1)^dn}^{(2)}(\gamma_k)  \right| \leq \sum_{\{x,y\} \subset \gamma_k} |g^{(2)}|(x-y) + \frac 32 (2k+1)^dn \int_{\RRd} |g^{(2)}|(x) dx.
\]
The points of a perturbed lattice lie in different boxes and the same goes for the points of~$\gamma_k$. For $\delta < 1/2$, the distance between two adjacent boxes is greater than $1/4$. Using the fact that that $g^{(2)}$ is decreasing and positive, it implies that
\[
	\sum_{\{x,y\} \subset \gamma_k} |g^{(2)}|(x-y)  \leq \frac{1}{2}(2k+1)^dn  \sum_{u\in\ZZd} 4^s |g^{(2)}|(u), 
\]
which concludes the expected control for $\Hback {(2k+1)^dn}^{(2)}(\gamma_k)$ and the proof of the proposition.
\end{proof}

\subsubsection{Control of the partition function}

\begin{proof}[Proof of Lemma \ref{L:fonction.part}]
For the upper-bound, the stability property of the energy gives
\[
Z^\beta_n  \leq \int e^{-\beta A n} \Bin {\Lambda_n} n(d\gamma)  =  (e^{-\beta A})^{n}.
\]
For the lower-bound, introducing the subset of perturbed lattice configurations defined in the previous section, that leads to 
\[
	Z^\beta_n \geq \int_{\mathcal{C}_{\delta,n}} e^{-\beta \Hn(\gamma)} \Bin {\Lambda_n} n (d\gamma). 
\]
According to Proposition~\ref{P:energy.perturbed.lattice}, we have $Z^\beta_n \geq e^{-\beta cn} \Bin {\Lambda_n} n (\Conf_{\delta,n})$. The next lemma, whose proof is postponed below, gives a lower-bound of the volume of the perturbed lattice configurations under the law of the binomial point process. 
\begin{lem}\label{L:perturbed.lattice.bin} We can choose $\delta >0 $ small enough such that for every $n\geq 1$
\[
	\Bin {\Lambda_n} n(\mathcal{C}_{n,\delta}) \geq n! \left( \frac{\delta}{n}\right)^n.
\]
\end{lem}
Applying the lemma for $\delta$ chosen small enough, and using the inequality $n! / n^n \geq e^{-n}$, that gives 
\[
	Z_n^\beta \geq \left(\delta e^{-(\beta c + 1)}\right)^n.
\]
\end{proof} 
\begin{proof}[Proof of Lemma \ref{L:perturbed.lattice.bin}]

To construct a perturbed lattice configuration, we start by spread $r_n^d$ points in the boxes $\Delta_j$ with $j = 1, ..., r_n^d$. Since they have a volume of $\delta^d$, it corresponds to an event of probability

\[
	p_ 1 = (r_n^d)! \left(\frac{\delta}{n}\right)^{r_n^d}.
\]

Then we have to spread the $n-r_n^d$ remaining points between the boxes $\Delta_j$ where ${j=r_n^d+1, ..., (r_n+2)^d}$, which is an event of probability $p_2$. We consider two cases. 

\textit{1.} $n^{1/d} - r_n \geq \delta$,  all the boxes have a volume equal to $\delta^d$ and we have the lower-bound  
\[
	p_2 \geq (n - r_n^d)! \left(\frac{\delta}{n}\right)^{n - r_n^d}.
\]

\textit{2.} $n^{1/d} - r_n \leq \delta$, the volume of the boxes is then smaller than $\delta^d$, we try to put the points in the "good boxes", those with a volume equal to  $\delta^{d-1}(n^{1/d} - r_n)$. There are $dr_n^{d-1}$ such boxes, and if $\delta \leq d2^{-d}$
\[
n - r_n^d \leq \sum_{k=1}^d \delta^k r_n^{d-k} \binom{d}{k} \leq \delta r_n^{d-1} 2^d \leq dr_n^{d-1},
\]
hence we can spread the remaining points among the good boxes. This gives a lower-bound of the probability $p_2$ 
\[
p_2 \geq(n - r_n^d)! \binom{dr_n^{d-1}}{n-r_n^d} \left(\frac{(n^{1/d} - r_n)\delta^{d-1}}{n}\right)^{n-r_n^d}.
\]
Using successively the inequalities
 $\binom{N}{p} \geq (N/p)^p$ and  $b^d - a^d \leq d(b-a)b^{d-1}$ (for $a < b$), that leads to  
\[
p_2 \geq (n - r_n^d)!  \left(\left(\frac{r_n}{n^{1/d}}\right)^{d-1} \frac{\delta^{d-1}}{n}\right)^{n-r_n^d}
\]
Since $r_n / n^{1/d} \geq 1/2$, for $\delta \leq (1/2)^{d-1}$ we have 
\[
	p_2 \geq (n - r_n^d)! \left(\frac{\delta}{n}\right)^{n - r_n^d}.
\]
At the end, if $\delta$ is chosen small enough
\[
\Bin {\Lambda_n} n(\mathcal{C}_{n,\delta}) = \binom{n}{r_n^d} p_1 p_2 \geq n! \left( \frac{\delta}{n}\right)^n.
\]
\end{proof}

\subsubsection{Existence of an accumulation point}
Before starting the proof of Proposition~ \ref{P:accumulation.points}, we introduce the specific entropy criterion which is the compactness tool we use. We recall that for two probability measures $\mu, \nu$ on the same measurable space, the relative entropy of $\mu$ with respect to $\nu$ is defined as 
\[
    I(\mu \mid \nu) 
    = 
    \left\{
    \begin{array}{ll}
        \displaystyle \int \log\left( \frac{d\mu}{d\nu}\right) d\mu & \text{if } \mu  \ll \nu,  \\ \\
         +\infty & \mbox{otherwise.}
    \end{array}
    \right.
\]

\begin{defin}\label{D:specific.entropy}
 Let $\P$ be a stationary probability measure on $\Conf$, the specific entropy of~$\P$ (with respect to $\Pi$) is defined by
\[
    I(\P) 
    =  \lim_{n\rightarrow +\infty} \frac 1{\leb(\Lambda_n)} I(\P_{\Lambda_n} \mid \Pi_{\Lambda_n}).
\]
\end{defin}
The proof of the existence of the limit in the previous definition and other properties of the specific entropy can be found in the Chapter 15 of \cite{georgii.book}. The following result, stated as Proposition~2.6 in \cite{georgii.zessin}, is the compactness criterion. 
\begin{prop}\label{P:tension}
For any $c > 0$, the set of probability measures 
$\{ \P \in \mathcal{P}_{\sta}(\Conf) : I(\P) \leq c\}
$
is compact and sequentially compact for the local topology. 
\end{prop}

\begin{proof}[Proof of Proposition~ \ref{P:accumulation.points}] 
The specific entropy criterion is adapted for stationary point processes, so we need to construct a stationarized version of $(\Pbn)_{n\geq 1}$. Let $(\P^{\beta,\per}_n)_{n\geq 1}$ be the point process consisting of independent copies of $\Pbn$ in the cubes $\Lambda_n + n^{1/d}u$ with $u\in\ZZd$. Formally we can define it as follow, if $A$ is a local event and $k$ an integer such that $A \in \F_{\Lambda_{(2k+1)^d n}}$ then
\[
	\P_n^{\beta,\per}(A) = \int 1_A\Big(\bigcup_{u\in[k]} \tau_u(\gamma_u) \Big) \prod_{u\in[k]} \Pbn(d\gamma_u).
\]
The probability measure $\P^{\beta,\per}_n$ is an infinite volume point process, but not a stationary one yet. The stationarized version of $\Pbn$ is defined by 
\[\mathbb{P}^{\beta, \sta}_n = \frac{1}{n} \int_{\Lambda_n} \mathbb{P}^{\beta,\per}_n \circ \tau_u du.
\]
To show that the sequence $(\P^{\beta,\sta}_n)_{n\geq 1}$ admits a convergent subsequence, we need an uniform bound of the specific entropy. According to the definition of the specific entropy and the properties of the relative entropy we have  
\[	
	I(\mathbb{P}^{\beta,\sta}_n) \leq \frac 1n I(\mathbb{P}^\beta_n \mid \Pi_{\Lambda_n}).
\]
We compute the relative entropy of $\Pbn$ with respect to $\pi_{\Lambda_n}$ as follows
\begin{align*}
	I_{\Lambda_n}(\Pbn \mid \pi) & = \int \log\left(\frac{d\Pbn}{d\Pi_{\Lambda_n}}\right) d\Pbn \\
	& = \int \log\left(\frac{d\Pbn}{d \Bin {\Lambda_n} n}\right) d\Pbn + \int \log\left(\frac{d \Bin {\Lambda_n} n}{d\Pi_{\Lambda_n}}\right) d\mathbb{P}^\beta_n \\
	& = - \log(Z^\beta_n) - \beta \int \Hn d\Pbn + \log\left(e^{n} \frac{n!}{n^n}\right).
\end{align*}
Using the stability property of $\Hn$ and the lower-bound on the partition function given by Lemma \ref{L:fonction.part} we obtain 
\[
	\frac 1n I(\Pbn \mid \Pi_{\Lambda_n}) \leq b_\beta - \beta A + 1. 
\]
According to Proposition~\ref{P:tension}, the sequence $(\mathbb{P}^{\beta, \sta}_n)_{n\geq 1}$ admits a convergent subsequence for the local topology, denoted by $(\mathbb{P}^{\beta, \sta}_{n_j})_{j\geq 1}$. Let $\Pb$ be the limit of this subsequence, it is a stationary probability measure on $\Conf$. Since all probability measures $\mathbb{P}^{\beta, \sta}_n$ have intensity equal to $1$, the converge for the local topology ensures that $\Pb$ has also an intensity equal to~$1$. To conclude the proof, it remains to show that $\Pb$ is the limit of $(\Pbnj)_{j\geq 1}$ as well.  Without loss of generality, we assume that $n_j = j$, or in other words that $\Pb$ is the limit of $(\P^{\beta,\sta}_n)_{n\geq 1}$. Let $f$ be a local tame function, and $\Delta$ a Borel bounded subset of $\RRd$ such that for all point configurations ~$\gamma$ we have $f(\gamma) = f(\gamma_\Delta)$. By definition of $\Pb$ 
\[
	\int f(\gamma) \mathbb{P}(d\gamma) =\lim_{n\rightarrow +\infty} \frac{1}{\leb(\Lambda_{n})} \int_{\Lambda_n} \int f(\tau_u(\gamma)) \mathbb{P}^{\beta, \per}_n(d\gamma) du.
\]
If we take $n$ large enough to have $\Delta \subset \Lambda_n$ and introduce $\Lambda_n^\star = \{u \in \Lambda_n : \tau_u(\Delta) \subset \Lambda_n\}$, then for all $u \in \Lambda_n^\star$ we have
\[
\int f(\tau_u(\gamma)) \mathbb{P}^{\beta, \per}_n(d\gamma) = \int f(\tau_u(\gamma)) \mathbb{P}^{\beta}_n(d\gamma).
\]
The key remark here is that $\Bin {\Lambda_n} n \circ \tau_{-u}^n$ is equal to $\Bin {\Lambda_n} n$, where $\tau_{u}^n$ is the translation of vector $u$ on $\Lambda_n$ seen as a torus. Since $\Hn$ is invariant by such torus translations,  for every $u \in \Lambda_n^\star$ 
\[
\int f(\tau_u(\gamma)) \mathbb{P}^{\beta}_n(d\gamma) 
=
\int f(\tau_u^n(\gamma)) \mathbb{P}^{\beta}_n(d\gamma)
=\int f(\gamma) \mathbb{P}^{\beta}_n(d\gamma).
\]
Finally, since $\leb(\Lambda_n^\star) / \leb(\Lambda_n)$ goes to $1$ as $n$ goes to infinity, we obtain  successively
\begin{align*}
\int f(\gamma) \mathbb{P}(d\gamma)  
& =\lim_{n\rightarrow +\infty} \frac{1}{\leb(\Lambda_n)} \int_{\Lambda^*_n} \int f(\tau_u(\gamma)) \mathbb{P}^{\beta, \per}_n(d\gamma) du\\
& =  \lim_{n\rightarrow +\infty} \frac{\leb(\Lambda_n^\star)}{\leb(\Lambda_n)} \int f(\gamma) \mathbb{P}^{\beta}_n(d\gamma)\\
& = \lim_{n\rightarrow +\infty} \int f(\gamma) \mathbb{P}^{\beta}_n(d\gamma).
\end{align*}
That proves that any accumulation point of the sequence $(\P_n^{\beta,\sta})_{n\geq 1}$ is also an accumulation point of the sequence $(\Pbn)_{n\geq 1}$.
\end{proof}

\subsection{Canonical DLR equations}

In this section we provide the proof of Theorem~\ref{T:DLR.equations} inspired by  part (C) of Theorem~1.1 in \cite{DLRsinebeta}. We start by proving the existence of the Move functions for point processes with bounded intensity. Then the theorem is proved using truncations of these Move functions, local convergence and canonical DLR equations in finite volume. 

\subsubsection{Move functions and their truncations in infinite volume}
We prove Proposition~\ref{P:def.Hn} and derive a control lemma which is needed for the proof the canonical DLR equations in infinite volume.
Let $\Delta$ be a bounded Borel set, $p\geq 1$ an integer and $\delta > 0$ a real. We introduce the event 
\[
A_{\Delta}^{(p)}(\delta) = \left\{\gamma \in \Conf : \sup_{\substack{\eta \in\Conf_\Delta\\ \vert \eta \vert = \vert \gamma_\Delta\vert} } \left\vert M_\Delta(\eta, \gamma)- M_\Delta^{(p)}(\eta, \gamma)\right\vert \leq \delta \right\}.
\]
\begin{lem}\label{L:Move.control}
Let $\kappa > 0$ be a fixed parameter, for every $\varepsilon > 0$ and every $\delta > 0$,
there exists $p >0$ such that for every point process $\mathbb{P}$ with intensity bounded by $\kappa$ we have
\[
\mathbb{P}\left(A_{\Delta}^{(p)}(\delta)\right) \geq 1 - \varepsilon.
\]
\end{lem}
\begin{proof}[Proof of Proposition \ref{P:def.Hn} and Lemma \ref{L:Move.control}]  Let $\mathbb{P}$ be a point process of intensity bounded by $\kappa$. For $p\geq1$, $M_\Delta^{(p)}(\eta,\gamma) =  +\infty$ if and only if $0 \in \gamma_{\Delta^c}$, which is an event of zero probability under $\mathbb{P}$. Therefore, the sequence $(M^{(p)}_\Delta(\eta,\gamma))_{p\geq 1}$ contains $\mathbb{P}$-almost surely finite elements. For integers $q > p > 0$ we have 
\begin{align*}
\left\vert M^{(q)}_\Delta(\eta, \gamma)- M^{(p)}_\Delta(\eta, \gamma)\right\vert  \leq \sum_{x \in \eta_\Delta} \sum_{y \in \gamma_{\Lambda_q \setminus \Lambda_p}} \vert g(y-x) - g(y) \vert
\end{align*}
Using the gradient of the potential $g$
\[
\vert g(y-x) - g(y) \vert = \left\vert \int_0^1 \langle \nabla g(y - tx), x \rangle dt \right\vert
\leq s\Vert x \Vert \int_0^1 \Vert y - tx \Vert^{-(s+1)}dt.
\] 
We introduce the quantity $\rho_\Delta = \max_{x\in\Delta} \Vert x \Vert$. For $p$ large enough such that $\rho_\Delta \leq p^{1/d} /4$, for every $y \in \Lambda_p^c$ and $x \in \Delta$
\[
\Vert y - tx \Vert \geq \Vert y \Vert - \Vert x \Vert \geq \Vert y \Vert / 2.
\]
Thus we obtain successively 
\begin{align*}
\left\vert M^{(q)}_\Delta(\eta, \gamma)- M^{(p)}_\Delta(\eta, \gamma)\right\vert  & \leq s2^{s+1}\sum_{x \in \eta_\Delta} \sum_{y \in \gamma_{\Lambda_q \setminus \Lambda_p}} \Vert x \Vert \Vert y \Vert^{-(s+1)} \\
& \leq s 2^{s+1} \rho_\Delta N_\Delta(\eta) \sum_{k = p}^{q-1} \sum_{y \in \gamma_{\Lambda_{k+1} \setminus \Lambda_{k}}} \Vert y \Vert^{-(s+1)}\\
& \leq s 4^{s+1} \rho_\Delta N_\Delta(\eta) \sum_{k = p}^{+\infty} N_{\Lambda_{k+1} \setminus \Lambda_{k}}(\gamma) k^{-(s+1)/d}.
\end{align*}
Since $\leb(\Lambda_{k+1} \setminus \Lambda_k) = 1$ and $\P$ has the intensity bounded by $\kappa$, we have
\begin{equation}\label{E:expectation.cauchy}
E_{\P} \left[\sum_{k = p}^{+\infty}N_{\Lambda_{k+1} \setminus \Lambda_{k}}  k^{-(s+1)/d}\right] \leq \kappa \sum_{k=p}^{+\infty} k^{-(s+1)/d}.
\end{equation}
It implies that for $\P$-almost every $\gamma$, the sequence $(M^{(p)}_\Delta(\eta,\gamma))_{p\geq 1}$ is a Cauchy sequence which gives the existence of a finite limit
\[
M_{\Delta}(\eta,\gamma) = \lim_{p\rightarrow +\infty} M^{(p)}_\Delta(\eta, \gamma) . 
\]
Proposition~\ref{P:def.Hn} is proved, and it remains to prove Lemma \ref{L:Move.control}. The previous computation shows that 
\[
\sup_{\substack{\eta \in\Conf_\Delta\\ N_\Delta(\eta)  = N_\Delta(\gamma)} } \left\vert M_\Delta(\eta, \gamma)- M^{(p)}_\Delta(\eta, \gamma)\right\vert \leq  s 4^{s+1} \rho_\Delta N_\Delta(\gamma)\sum_{k = p}^{+\infty} N_{\Lambda_{k+1} \setminus \Lambda_{k}}(\gamma) k^{-(s+1)/d}.
\]
Since $\P$ has intensity bounded by $\kappa$, we can find a positive constant $K$ (depending only on $\kappa$ and $\Delta$) such that $\P(N_\Delta > K)\leq \varepsilon/2$. According to (\ref{E:expectation.cauchy}) and Markov's inequality, we can find $p > 0$ such that 
\[
	\P\left( s 4^{s+1} \rho_\Delta K \sum_{k = p}^{+\infty} N_{\Lambda_{k+1} \setminus \Lambda_{k}}  k^{-(s+1)/d} > \delta \right) \leq \varepsilon/2. 
\]  
For that choice of $p$, $\P(A^{(p)}_\Delta(\delta)) \geq 1-\varepsilon$, and $p$ depends only on $d,s,\delta,\Delta$ and $\kappa$.
\end{proof}

The sequence $(M_\Delta^{(p)}(\eta,\gamma))_{p \geq 1}$ allows to approximate $M_\Delta(\eta,\gamma)$, so we introduce the notation  
\[
	f^{(p)}_\Delta(\gamma) = \frac{1}{Z^{(p)}_\Delta(\gamma)} \int f(\eta \cup \gamma_{\Delta^c}) \exp\big(-\beta\big(H(\eta) + M^{(p)}_\Delta(\eta,\gamma)\big)\big) \Bin {\Delta} {N_\Delta(\gamma)}(d\eta),
\]
with the corresponding partition function 
\[
	Z^{(p)}_\Delta(\gamma)=  \int \exp\big(-\beta\big(H(\eta) + M^{(p)}_\Delta(\eta,\gamma)\big)\big) \Bin {\Delta} {N_\Delta(\gamma)}(d\eta).
\]
This partition function is positive since we integrate a $\Bin {\Delta} {N_\Delta(\gamma)}$-almost surely positive quantity, and is finite for $\Pb$-almost every $\gamma$ since we have the upper bound 
\[
	Z^{(p)}_\Delta(\gamma) \leq \exp\left(-n \beta \inf_{x\in\Delta} M_\Delta^{(p)}(\{x\}, \gamma)\right).
\]
The local function $f^{(p)}_\Delta$ gives an approximation of the non-local function $f_\Delta$, as described in the following result. 
\begin{prop}\label{P:DLRp.control}
For every $\varepsilon > 0$, there exists an integer $p > 0$ such that for any point process $\mathbb{P}$ with intensity bounded by $1$ and any local bounded function $f$ we have 
\[
\left\vert \E \P\big(f^{\vphantom{(p)}}_\Delta - f^{(p)}_\Delta\big) \right\vert \leq \varepsilon \normf.
\]
\end{prop}
\begin{proof}
We write the difference as 
\begin{align}\label{E:diff.functions}
& f_\Delta^{(p)}(\gamma) - f^{}_\Delta(\gamma) = \nonumber\\ 
&\quad \quad \quad \quad \quad 
 \frac{1}{Z^{(p)}_\Delta(\gamma)} \int f(\eta \cup \gamma_{\Delta^c}) \left(e^{-\beta(H(\eta)+ M^{(p)}_\Delta(\eta,\gamma))} - e^{-\beta (H(\eta) + M_\Delta(\eta,\gamma))}\right) \Bin {\Delta}{N_\Delta(\gamma)}(d\eta)  \nonumber\\ 
& \quad \quad \quad \quad \quad \quad
+ \left(\frac{1}{Z^{\vphantom{(p)}}_\Delta(\gamma)}  - \frac{1}{Z^{(p)}_\Delta(\gamma)} \right)\int f(\eta \cup \gamma_{\Delta^c}) e^{-\beta(H(\eta) + M_\Delta(\eta,\gamma))} \Bin {\Delta}{N_\Delta(\gamma)}(d\eta).
\end{align}
For $\delta > 0$, if $\gamma\in A^{(p)}_\Delta(\delta)$, then for every  $\eta \in \Conf_\Delta$,
\[
\left\vert e^{-\beta(H(\eta)+ M^{(p)}_\Delta(\eta,\gamma))} - e^{-\beta (H(\eta) + M_\Delta(\eta,\gamma))}\right\vert \leq (e^{\beta \delta} - 1) e^{-\beta(H(\eta)+ M^{(p)}_\Delta(\eta,\gamma))}.
\]
We can control the difference between the partition functions
\[
\left\vert Z^{(p)}_\Delta(\gamma) - Z_\Delta(\gamma) \right\vert \leq (e^{\beta\delta} - 1)Z^{(p)}_\Delta(\gamma).
\]
These two inequalities applied to (\ref{E:diff.functions}) give that for $
\gamma \in A_\Delta^{(p)}(\delta)$
\[
\left\vert f_\Delta^{(p)}(\gamma) - f_\Delta^{\vphantom{(p)}}(\gamma)\right\vert \leq 2(e^{\beta \delta} - 1) \normf.
\]
Taking expectations leads to 
\[
\left\vert E^{}_{\P}(f^{\vphantom{(p)}}_\Delta - f^{(p)}_\Delta) \right\vert \leq 2(e^{\beta \delta} - 1) \normf + 2(1 -  \P(A^{(p)}_\Delta(\delta)))  \normf. 
\]
We choose $\delta$ small enough in order to have 
$2(e^{\beta \delta} - 1) \leq \varepsilon / 2$. According to Lemma \ref{L:Move.control}, there exists $p>0$ such that if $\P$  has intensity bounded by 1 then $ \P(A^{(p)}_\Delta(\delta)) \geq 1 -  \varepsilon / 4$, which concludes the proof of the proposition.
\end{proof}

\subsubsection{DLR equations in finite volume}

In finite volume the canonical DLR  equations are simpler to obtained. We will state them using the local energy. For $\Delta \subset \Lambda_n$ a bounded Borel set and $\eta$, $\gamma$ two point configurations in $\Lambda_n$, the local energy of the points of $\eta$ in $\Delta$ into the field created by $\gamma$ is given by \[
	H_{n,\Delta}(\eta,\gamma) = \Hn(\eta_\Delta \cup \gamma_{\Delta^c}) - \Hn(\gamma_{\Delta^c}) = \Hn(\eta_\Delta) + \sum_{x\in\eta}\sum_{y\in\gamma_{\Delta^c}} g_n(x-y).
	\]
	It corresponds to the sum of the energy of $\eta_{\Delta}$ and the interaction between $\eta_{\Delta}$ and $\gamma_{\Delta^c}$. As before, for a measurable bounded function $f$ we introduce the notation 
	\[
		f_{n,\Delta}(\gamma) = \frac{1}{Z_{n,\Delta}(\gamma)}\int f(\eta \cup \gamma_{\Delta^c}) e^{-\beta H_{n,\Delta}(\eta, \gamma)} \Bin {N_\Delta(\gamma)} \Delta(d\eta), 
	\]
with the partition function 
\[
	Z_{n,\Delta}(\gamma) = \int e^{-\beta H_{n,\Delta}(\eta,\gamma)} \Bin \Delta {N_\Delta(\gamma)}(d\eta).
\]
We can now state the canonical DLR equations in the finite volume case. 
\begin{prop}\label{P:DLR.finite.vol}
Let $f$ be a bounded measurable function on $\Conf$ and $\Delta$ a bounded Borel subset of $\RRd$, if $n\geq 1$ is such that $\Delta \subset \Lambda_n$ then  
\[
	\E \Pbn(f) = \E \Pbn(f_{n,\Delta}). 
\] 
\end{prop}

\begin{proof}
Introducing the number of points in $\Delta$, the expectation can be written as  
\[
E_{\Pbn}(f) = \frac{1}{Z_n^\beta} \sum_{k=0}^n \int 1_{N_\Delta=k}(\gamma)  f(\gamma) e^{-\beta H_n(\gamma)} \Bin {\Lambda_n} n (d\gamma).
\]
We do the change of variables $\gamma = \eta \cup \xi$ with $(\eta, \xi) \in \Conf_{\Lambda_n\setminus \Delta} \times \Conf_{\Delta}$ and use the notation $p_\Delta = \leb(\Delta) / \leb(\Lambda_n)$ in order to obtain
\begin{equation}\label{E:DLRn}
E_{\Pbn}(f) = \frac{1}{Z^\beta_n} \sum_{k=0}^n \binom{n}{k} p_\Delta^k(1-p_\Delta)^{n-k} \iint f(\eta\cup\xi) e^{-\beta H_n(\eta\cup\xi)}\Bin {\Delta} k(d\eta) \Bin {\Lambda_n\setminus\Delta}{n-k}(d\xi).
\end{equation}
Introducing the partition function $Z_{n,\Delta}(\eta \cup \xi)$ in the integral, that leads to 
\begin{multline*}
\iint f(\eta\cup\xi) e^{-\beta H_n(\eta\cup\xi)}\Bin {\Delta} k(d\eta) \Bin {\Lambda_n\setminus\Delta}{n-k}(d\xi) \\ = \iiint \frac{1}{Z_{n,\Delta}(\eta\cup\xi)} f(\eta\cup\xi) e^{-\beta(H_n(\eta\cup\xi) + H_{n,\Delta}(\zeta\cup\xi))} \Bin {\Delta}{N_\Delta(\eta)}(d\zeta) \B_{\Delta,k}(d\eta) \Bin {\Lambda_n\setminus\Delta} {n-k}(d\xi).
\end{multline*}
Since $H_n(\eta\cup\xi) + H_{n,\Delta}(\zeta,\xi) = H_{n,\Delta}(\eta,\xi) + H_n(\zeta\cup\xi)$ and $Z_{n\Delta}(\eta\cup\xi) =Z_{n,\Delta}(\zeta\cup\xi)$ the integral can be written using the function $f_{n,\Delta}$ as 
\begin{multline*}
\iint f(\eta\cup\xi) e^{-\beta H_n(\eta\cup\xi)}\Bin {\Delta} k(d\eta) \Bin {\Lambda_n\setminus\Delta}{n-k}(d\xi) \\ = \iint f_{n,\Delta}(\zeta\cup\xi) e^{-\beta H_n(\zeta\cup\xi)}\Bin {\Delta} k(d\zeta) \Bin {\Lambda_n\setminus\Delta}{n-k}(d\xi).
\end{multline*}
Inserting the previous equality in (\ref{E:DLRn}), that gives that $E_{\Pbn}(f) = E_{\Pbn}(f_{n,\Delta})$.
\end{proof}
In infinite volume, canonical DLR equations involve Move functions, we can also define them in the finite volume case. Let $\Delta$ be a bounded Borel set, $\eta$ and $\gamma$ two points configurations in $\Lambda_n$, the cost of moving the points of $\eta$ in $\Delta$ from the origin  into the field created by $\gamma$ is given by 
\[
	M_{n,\Delta}(\eta, \gamma) = \sum_{x\in\eta_\Delta} \sum_{y\in\gamma_{\Delta^c}} \left[g_n(x-y) - g_n(y)\right].
\]
Since the difference between $H_{n,\Delta}(\eta,\gamma)$ and $H_n(\eta_{\Delta}) + M_{n,\Delta}(\eta,\gamma)$ depends only on $\gamma_{\Delta^c}$, the finite volume DLR equations can also involve the Move function by writing
\[
f_{n,\Delta}(\gamma) = \frac{1}{Z'_{n,\Delta}(\gamma)} \int f(\eta\cup\gamma_{\Delta^c}) \exp(-\beta(\Hn(\eta) + M_{n,\Delta}(\eta,\gamma)))\Bin \Delta {N_\Delta(\gamma)}(d\eta),
\]
with the new the partition function 
\[
Z'_{n,\Delta}(\gamma) = \int \exp(-\beta(\Hn(\eta) + M_{n,\Delta}(\eta,\gamma)))\Bin {\Delta}{N_\Delta(\gamma)}(d\eta).
\]
We need to control the error of replacing $f_{n,\Delta}$ by $f_\Delta$ in Proposition~\ref{P:DLR.finite.vol}, and start by a control of energies and Move functions. For $\delta > 0$, we introduce the event
\[
A_{n,\Delta}(\delta) = \left\{\gamma \in \Conf_{\Lambda_n} : \sup_{\substack{\eta \in\Conf_\Delta\\ N_\Delta(\eta) = N_\Delta(\gamma)}} \left\vert H(\eta) - H_n(\eta) + M_\Delta(\eta, \gamma)- M_{n,\Delta}(\eta, \gamma)\right\vert \leq \delta \right\}.
\]
\begin{lem}\label{L:energy.control.n}
For every $\varepsilon > 0$ and $\delta > 0$, there exists a integer $n_0 \geq 1$ such that for every $n\geq n_0$ we have 
\[
	\Pbn(A_{n,\Delta}(\delta)) \geq 1 - \varepsilon.
\]
\end{lem}
\begin{proof}
According to Lemma \ref{L:majoration.n}, for $n \geq 1$ large enough and for every $\eta \in \Conf_\Delta$,
\[
\vert H_n(\eta) - H(\eta) \vert \leq c N_\Delta(\eta)^2 n^{-s/d}
\; \text{and} \;
\left\vert M_\Delta(\eta,\gamma) - M_\Delta(\eta, \gamma) \right\vert \leq c_\Delta N_\Delta(\eta) N_{\Delta^c}(\gamma) n^{-(s+1)/d}.
\]
All the probability measures $(\Pbn)_{n\geq1}$ have intensity bounded by 1, hence we there exists $K > 0$ such that for every $n\geq 1$, $\Pbn(N_\Delta \leq K) \geq 1 -\varepsilon$. For an integer $n_0$ such that 
\[
	cK^2 n_0^{-s/d} + c_\Delta K n_0^{-(s+1-d)/d} \leq \delta,
\] 
and for every $n\geq n_0$, $\Pbn(A_{n,\Delta}(\delta)) \geq 1 - \varepsilon$.
\end{proof}

\begin{prop}\label{P:DLRn.control}
For every $\varepsilon > 0$ there exists an integer $n_0 \geq 1$ such that for every $n \geq n_0$ and every bounded local function $f$ we have
\[
\left\vert E_{\Pbn}(f_\Delta - f_{n,\Delta}) \right\vert \leq \varepsilon \Vert f \Vert_{\infty}.
\]
\end{prop}
\begin{proof}
The proof is similar the one of Proposition~\ref{P:DLRp.control}. 
\end{proof}

\subsubsection{Proof of Theorem \ref{T:DLR.equations}}
Let $\Pb$ be any limit for the local convergence of a subsequence $(\Pbnk)_{k\geq 1}$, $\Delta$ a bounded Borel subset of $\RRd$ and $f$ a local bounded function, we have 
\begin{align}
\left\vert E_{\Pb}(f) - E_{\Pb}(f_\Delta)\right\vert  \leq &  \left\vert E_{\Pb}(f) - E_{\Pbnk}(f)\right\vert \tag{a}\label{E1} \\ 
& + \left\vert E_{\Pbnk}(f) - E_{\Pbnk}(f_{n_k,\Delta})\right\vert  \tag{b}\label{E2}\\ 
& + \left\vert E_{\Pbnk}(f_{n_k,\Delta}) - E_{\Pbnk}(f_\Delta)\right\vert  \tag{c}\label{E3}\\ 
& + \left\vert E_{\Pbnk}(f_\Delta) - E_{\Pbnk}(f^{(p)}_\Delta)\right\vert \tag{d}\label{E4}\\ 
& + \left\vert E_{\Pbnk}(f^{(p)}_\Delta) - E_{\Pb}(f^{(p)}_\Delta)\right\vert \tag{e}\label{E5} \\ 
& + \left\vert E_{\Pb}(f^{(p)}_\Delta) - E_{\Pb}(f_\Delta)\right\vert \tag{f}.\label{E6}
\end{align}
The local convergence of $(\Pbnk)_{k\geq 1}$ to $\Pb$ ensures that (\ref{E1}) and (\ref{E5}) goes to zero when $n_k$ goes to infinity. The canonical DLR equation in finite volume (Proposition~ \ref{P:DLR.finite.vol}) tells us that (\ref{E2}) is null. We can treat (\ref{E3}) with Proposition~ \ref{P:DLRn.control} and (\ref{E4}),(\ref{E6}) with Proposition~\ref{P:DLRp.control}. At the end, for every $\varepsilon > 0$ we have 
\[
	\left\vert E_{\Pb}(f) - E_{\Pb}(f_\Delta)\right\vert  \leq \varepsilon \Vert f \Vert_{\infty},
\]
which proves Theorem \ref{T:DLR.equations} for any local and bounded function $f$. The result is classically extended to all measurable, bounded functions by a monotone class argument.

\subsection{Existence of a non number-rigid Riesz gas}
In this section we prove Theorem \ref{T:point.deletion}. The $\beta$-Circular Riesz Gas $\Pbstar$ is constructed as an accumulation point of a particular subsequence along which the expectation of local energy is uniformly bounded. It allows to control the cost of exchanging the point configuration in a compact $\Delta$ with the one in the translated $\Delta + u$, where $\Vert u \Vert$ can be as large as we want. Then the idea is to let $u$ goes to infinity in order to prove that the number of points in $\Delta$ can be as we want.

\subsubsection{Construction of $\Pb_\star$}
We use a subsequence for which we control uniformly the expectation of the local energy, defined for a configuration $\gamma \in \Conf_{\Lambda_n}$  and $x\in \Lambda_n$ by
\[
	h_n(x,\gamma) = \sum_{y\in\gamma} g_n(x-y).
\]
The existence of a such subsequence is assured by the following result.
\begin{prop}
For every $\beta > 0$, there exists a constant $K> 0$ and a an increasing sequence $(n_j)_{j\geq 1}$ such that for all $j\geq 1$ 
\[
	\E {\Pb_{n_j}} (\vert h_{n_j}(0,.) \vert) \leq K.
\]
\end{prop}
According to Proposition~\ref{P:accumulation.points}, the sequence $(\Pb_{n_j})_{j\geq 1}$ admits an accumulation point, denoted $\Pb_\star$, which is our non number-rigid candidate.
\begin{proof}Since $\Pbn$ is invariant by torus translations on $\Lambda_n$, for every bounded Borel subset $\Delta\subset\Lambda_n$, $E_{\Pbn}(N_\Delta) = \leb{}(\Delta)$ and 
\[
\E \Pbn(h_n(0,\cdot)) = \int \sum_{x\in\gamma} g_n(x) \Pbn(d\gamma) = \int g_n(x) \leb_{\Lambda_n}(dx) = 0. 	
\]
Therefore, we only need to find a subsequence such that $$\E \Pbnj(h^-_n(0,\cdot)) \geq -K$$ where we use the notation $h^-_n(0,\gamma) = \min(h^-_n(0,\gamma), 0)$. The following lemma is useful to manage computations under $\Pbn$. Its proof is given at the end of the section. 
\begin{lem}\label{L:GNZ.finite.volume} For any measurable function $f : \RRd\times \Conf \rightarrow\RR$
\[
	\int \sum_{x\in\gamma} f(x,\gamma\setminus \{x\})\Pbn(d\gamma) = \frac{1}{Z_n^\beta} \iint f(x,\gamma) e^{-\beta H_n(\gamma \cup \{x\})} \Bin {\Lambda_n} {n-1}(d\gamma) \leb_{\Lambda_n}(dx).
\]
\end{lem}
We introduce a modified inverse temperature $\beta_n = \beta (n/(n-1))^{s/d}$. The stability property of $\Hn$ (Proposition~\ref{P:properties.Hn}) and  Lemma \ref{L:GNZ.finite.volume} give
\begin{align*}
2A & \leq \frac{1}{n} E_{\mathbb{P}^{\beta_n}_{n}}(2H_n) \\
& = \frac 1n \int \sum_{x\in\gamma} h_n(x,\gamma\setminus x) \Pbn(d\gamma) \\
& = \frac{1}{n Z^{\beta_n}_n}\iint h_n(x,\gamma)e^{-\beta_n H_n(\gamma\cup \{x\})} \Bin {\Lambda_n} {n-1} (d\gamma) \leb_{\Lambda_n}(dx).
\end{align*}
For $x\in \Lambda_n$, we have $h_n(x,\gamma) = h_n(0,\toren_{-x}(\gamma))$ and $H_n(\gamma \cup \{ x\}) = H_n(\toren_{-x}\gamma \cup \{ 0\})$. Since $\Bin {\Lambda_n} {n-1}$ is invariant by the torus translation $\toren_{-x}$, the dependency in $x$ disappears and 
\[
	2A \leq  \frac{1}{Z_n^{\beta_n}} \iint h_n(0,\gamma) e^{-\beta_n H_n(\gamma \cup \{0\})} \Bin {\Lambda_n} {n-1}(d\gamma).
\]
To pass from a binomial point process on $\Lambda_n$ to one on $\Lambda_{n-1}$, we use the scale change of parameter $r_n = (n/(n-1))^{1/d}$, which leads to 
\[
2A  \leq \frac{1}{Z_n^{\beta_n}} \int h_n(0,r_n\gamma) e^{-\beta_n h_n(0,r_n \gamma)} e^{-\beta_n H_n(r_n \gamma)} \Bin {\Lambda_{n-1}} {n-1}(d\gamma),
\]
where $r_n \gamma = \{ r_n x,\; x \in \gamma\}$.
According the scaling property of the periodic potential $g_n$ we obtain
\[
 2A \leq r_n^{-s} \frac{1}{Z_n^{\beta_n}} \int h_{n-1}(0,\gamma) e^{-\beta h_{n-1}(0,\gamma)} e^{-\beta H_{n-1}(\gamma)} \Bin {\Lambda_{n-1}} {n-1}(d\gamma),
\]
and recover the finite volume Gibbs measure on $\Lambda_{n-1}$ with inverse temperature $\beta$ 
\[
2A \leq r_n^{-s} \frac{Z^\beta_{n-1}}{Z^{\beta_n}_n} E_{\mathbb{P}^\beta_{n-1}}\left(h_{n-1}(0,\cdot) e^{-\beta h_{n-1}(0,\cdot)}\right).
\]
Since $xe^{-\beta x} \leq x$ if $x \leq 0$ and $x e^{-\beta x} \leq e^{-1}/\beta$ if $x \geq 0$, we have 
\[
2A \leq r_n^{-s} \frac{Z^\beta_{n-1}}{Z^{\beta_n}_n}\left( E_{\mathbb{P}^\beta_{n-1}}\left(h^-_{n-1}(0,.)\right) + \frac{e^{-1}}{\beta}\right).
\]
The choice of a good subsequence $(n_j)_{j\ge 1}$ allows to bound from below the quotient of partition functions, as explained in the following result whose the proof is postponed as well. 
\begin{lem}\label{L:good.indices}
There are infinitely many indices $n\geq 2$ such that 
\[
	\frac{Z^\beta_{n-1}}{Z^{\beta_n}_n} \geq \frac{e^{4\beta A}}{2b_\beta},
\] 
where $b_\beta$ is the constant appearing in Lemma \ref{L:fonction.part}.
\end{lem}
The lemma implies that we can find a subsequence $(\Pb_{\phi(n)})_{n\geq 1}$ such that for every $n\geq 2$
\[
E_{\mathbb{P}^\beta_{n-1}}\left(h^-_{n-1}(0,.)\right) \geq - \frac{e^{-1}}{\beta} + 8A b_\beta e^{-\beta A s},
\]
which concludes the proof of the proposition. 
\end{proof}
\begin{proof}[Proof of Lemma \ref{L:GNZ.finite.volume}]
According to the definition of $\Pbn$ and using $\leb(\Lambda_n) =n$ we have 
\begin{multline*}
	\int \sum_{x\in\gamma} f(x,\gamma\setminus \{x\}) \Pbn(d\gamma) \\ = \frac 1 {n^n}\sum_{j=1}^n \int_{\Lambda_n^n}  f(x_j, \{x_1, ..., x_n\}\setminus\{x_j\})\frac{1}{\Zbn} e^{-\beta H_n(\{x_1, ..., x_n\})} dx_1 ... dx_n.
\end{multline*}
Since the terms in the summation do not depend on the index $j$ and are equal,
\begin{multline*}
	\int \sum_{x\in\gamma} f(x,\gamma\setminus \{x\}) \Pbn(d\gamma)\\ = \frac 1 {n^{n-1}}\int_{\Lambda_n^n}  f(x_n, \{x_1, ..., x_{n-1}\})\frac{1}{\Zbn} e^{-\beta H_n(\{x_1, ..., x_n\})} dx_1 ... dx_n.
\end{multline*}
We recover the definition of $\Bin {\Lambda_n} {n-1}$ and then obtain the result of Lemma \ref{L:GNZ.finite.volume}, with $x = x_n$ and $\gamma = \{x_1, ..., x_{n-1}\}$.  
\end{proof}
\begin{proof}[Proof of Lemma \ref{L:good.indices}]
Since $\beta_n = \beta (n/(n-1))^{s/d}$, for all $n\geq 2$ we have $0 \leq \beta_n - \beta \leq 4\beta /n$. Using the stability property of the energy, we bound from above the partial derivative of the partition function with respect to the inverse temperature
\[
\frac{\partial Z^\beta_n}{\partial\beta} = - \int \Hn e^{-\beta \Hn} d\Bin {\Lambda_n} n  \leq -A n Z^\beta_n. 
\]
Applying the Grönwall's inequality, that leads to  
\[
	Z_n^{\beta_n} \leq Z_n^\beta e^{-An(\beta_n - \beta)} \leq Z_n^\beta e^{-4\beta A/d}.
\]	
It remains to prove that there are infinitely many indices such that $Z^\beta_{n-1} / Z^\beta_n  \leq  1 / (2b_\beta)$. If we assume that there exists an index $n_0$ such that for all $n\geq n_0$, $Z_n^\beta > 2b_\beta Z^\beta_{n-1}$, then we have for all $n\geq n_0$, $Z_n^\beta > (2b_\beta)^{n-n_0} Z^\beta_{n_0}$, which is in contradiction with the upper-bound of the Lemma \ref{L:fonction.part}.
\end{proof}
\subsubsection{A key inequality}
\begin{prop}\label{P:inequality.Pbn}
For every pair of integers $(k,l)$, every compact $\Delta$ and $\varepsilon > 0$ there exists a positive constant $C^{k,l}_{\Delta, \varepsilon}$ such that for every bounded Borel set $\Lambda\subset \RRd$ containing $\Delta$, every $\F_{\Lambda\setminus \Delta}$-mesurable bounded function $f$, every $u\in \RRd$ such that $d(\Lambda, \Delta + u) > 1$ and every integer $j$ such that $\Lambda \cup (\Delta + u) \subset \Lambda_{n_j}$ 
\[
	\int 1_{N_\Delta = k}(\gamma) 1_{N_{\Delta + u} = l}(\gamma) f(\gamma) \Pbnj(d\gamma) \leq C^{k,l}_{\Delta, \varepsilon} \int 1_{N_\Delta = l}(\gamma) 1_{N_{\Delta+u} = k}(\gamma) f(\gamma) \Pbnj(d\gamma) + \varepsilon \normf.
\]
\end{prop}
It is crucial to note that the constant in front of the integral does not depend of $u$, and this is the case as soon as $\Vert u \Vert$ is large enough. In order to simplify the notations, during the proof, we assume  without loss of generality that $n_j = j$, or in other words, that there exists a constant $K\geq 0$ such that for all $n\geq 1$, we have $\E \Pbn(|h_n(0,\cdot)|) \leq K$.  
\begin{proof}
Let $\mathcal{A}_{n,u,\varepsilon}$ be an event such that $\Pbn(\mathcal{A}_{n,u,\varepsilon}) \geq 1 -\varepsilon$, its definition will be precised later, then  
\begin{multline*}
	\int 1_{N_\Delta = k}(\gamma) 1_{N_{\Delta + u} = l}(\gamma) f(\gamma) \Pbn(d\gamma) \\ \leq \int 1_{\mathcal{A}_{n,u,\varepsilon}}(\gamma)1_{N_\Delta = k}(\gamma) 1_{N_{\Delta + u} = l}(\gamma) f(\gamma) \Pbn(d\gamma) + \varepsilon \normf.
\end{multline*}

Applying the canonical DLR equation in finite volume (Proposition~\ref{P:DLR.finite.vol}) with the bounded set $V = \Delta \cup (\Delta + u)$ and the bounded function $1_{\mathcal{A}_{n,u,\varepsilon}}1_{N_\Delta = k} 1_{N_{\Delta +u} = l} f$ leads to
\begin{align*}
& \int 1_{\mathcal{A}_{n,u,\varepsilon}}(\gamma) 1_{N_\Delta = k}(\gamma) 1_{N_{\Delta + u} = l}(\gamma) f(\gamma) \Pbn(d\gamma) \\ 
 & = \iint  1_{\mathcal{A}_{n,u,\varepsilon}}(\eta\cup\gamma_{V^c}) 1_{N_\Delta = k}(\eta) 1_{N_{\Delta + u} = l}(\eta)1_{N_V = k+l}(\gamma) f(\gamma)e^{-\beta H_{n,V}(\eta, \gamma)} \Bin V{k+l}(d\eta) \Pbn(d\gamma).
\end{align*}
We focus on the integral with respect to $\eta$ and use the change of variables $\eta = \eta_1 \cup (\eta_2 + u)$ with $\eta_1,\eta_2 \in \Conf_\Delta$ 
\begin{multline*}
	\int  1_{\mathcal{A}_{n,u,\varepsilon}}(\eta\cup\gamma_{V^c}) 1_{N_\Delta = k}(\eta) 1_{N_{\Delta + u} = l}(\eta) e^{-\beta H_{n,V}(\eta, \gamma)} \Bin V{k+l}(d\eta) \\ = \binom{k+l}{k}\left(\frac 12\right)^{k+l} \iint 1_{\mathcal{A}_{n,u,\varepsilon}}(\gamma_{V^c} \cup \eta_1 \cup (\eta_2+u)) e^{-\beta H_{n,V}(\eta_1\cup(\eta_2+u), \gamma)} \Bin \Delta k(d\eta_1) \Bin {\Delta} l(d\eta_2).
\end{multline*}
At this point the idea is to switch the configurations between $\Delta$ and $\Delta + u$. We only need to control the corresponding change in the energy. It is achieved by the next lemma, whose the proof is postponed.

\begin{lem}\label{L:switch}
There exist a constant $C^{k,l}_{\Delta,\varepsilon}$ and for every $n$ large enough, an event $\mathcal{A}_{n,u,\varepsilon}$ with $\Pbn(\mathcal{A}_{n,u,\varepsilon}) \geq 1 - \varepsilon$, such that the following property holds: For every configurations $\eta_1,\eta_2 \in \Conf_\Delta$ and $\gamma \in \Conf_{\Lambda_n}$ such that $N_\Delta(\eta_1) = k$ and $N_\Delta(\eta_2) = l$, if $\gamma_{V^c}\cup\eta_1\cup(\eta_2+u) \in \mathcal{A}_{n,u,\varepsilon}$  then we have
\[
	 e^{-\beta H_{n,V}(\eta_1\cup(\eta_2+u), \gamma)} \leq C^{k,l}_{\Delta, \varepsilon}e^{-\beta H_{n,V}((\eta_1+u)\cup \eta_2, \gamma)}.
\]
\end{lem}
If we chose the events $\mathcal{A}_{n,u,\varepsilon}$ according to Lemma \ref{L:switch}, we obtain that 
\begin{multline*}
	\int  1_{\mathcal{A}_{n,\varepsilon}}(\gamma_{V^c}\cup \eta) 1_{N_\Delta = k}(\eta) 1_{N_{\Delta + u} = l}(\eta) e^{-\beta H_{n,V}(\eta, \gamma)} \Bin V{k+l}(d\eta) \\ \leq C^{k,l}_{\Delta, \varepsilon} \binom{k+l}{k}\left(\frac 12\right)^{k+l} \int  e^{-\beta H_{n,V}((\eta_1+u)\cup \eta_2), \gamma)} \Bin \Delta k(d\eta_1) \Bin {\Delta} l(d\eta_2).
\end{multline*}
It remains to do the previous steps backward, the change of variables $\eta = \eta_1 \cup \eta_2 \in \Conf_V$ leads to 
\begin{multline*}
\int  1_{\mathcal{A}_{n,u,\varepsilon}}(\gamma_{V^c}\cup \eta) 1_{N_\Delta = k}(\eta) 1_{N_{\Delta + u} = l}(\eta) e^{-\beta H_{n,V}(\eta, \gamma)} \Bin V{N_V(\gamma)}(d\eta) \\ \leq  C^{k,l}_{\Delta,\varepsilon} \int  1_{N_\Delta = l}(\eta) 1_{N_{\Delta + u} = k}(\eta) e^{-\beta H_{n,V}(\eta, \gamma)} \Bin V{N_V(\gamma)}(d\eta).
\end{multline*}
Using again Proposition \ref{P:DLR.finite.vol}, that gives at the end
\[
\int 1_{\mathcal{A}_{n,u,\varepsilon}}(\gamma) 1_{N_\Delta = k}(\gamma) 1_{N_{\Delta + u} = l}(\gamma) f(\gamma) \Pbn(d\gamma) \leq  C^{k,l}_{\Delta,\varepsilon} \int  1_{N_\Delta = l}(\gamma) 1_{N_{\Delta + u} = k}(\gamma) f(\gamma) \Pbn(d\gamma).
\]
\end{proof}
\begin{proof}[Proof of Lemma \ref{L:switch}]
For $u \in \RRd$ and $C > 0$, we introduce the event

\[
	A_{n,u}(C) = \left\{\gamma \in \Conf_n : \sup_{x\in\Delta+u} \left|h_n(x,\gamma_{(\Delta+u)^c})\right| \leq C(1 + N_{\Delta+u}(\gamma))\right\}
\]
\begin{lem}\label{L:control.proba.Pbn}
For every $\varepsilon > 0$, we can find a constant $C>0$ such that for every $u$ and every $n$ such that $\Delta + u \subset \Lambda_n$ we have $\Pbn(A_{n,u}(C)) \geq 1 - \varepsilon$. 
\end{lem}

Before proving the lemma, we show how the event $\mathcal{A}_{n,u,\varepsilon} = A_{n,0}(C)\cap A_{n,u}(C)$ satisfies the requests of Lemma \ref{L:switch}. According to Lemma \ref{L:control.proba.Pbn}, we can chose $C$ such that $\Pbn(A_{n,0}(C) \cap A_{n,u}(C)) \geq 1 - \varepsilon$ and according to Lemma \ref{L:majoration.n} we can add that for every $x_1,x_2 \in \Delta$ we have $\vert g_n(x_1-x_2 +u) \vert \leq C$ as soon as $d(\Lambda, \Delta + u) > 1$ and $n$ is large enough. It is important to notice that the constant $C$ depends only of $\Delta$ and $\varepsilon$. The local energy can be decomposed as
\[
	H_{n,V}(\eta_1 \cup (\eta_2+u), \gamma) = H_n(\eta_1 \cup (\eta_2+u)) + \sum_{x_1 \in\eta_1} h_n(x_1, \gamma^{}_{V^c}) +
	\sum_{x_2 \in(\eta_2+u)}  h_n(x_2, \gamma^{}_{V^c}).
\]
Since for $x \in \eta_1$ 
\[
	| h_n(x,\gamma_{V^c}) | \leq | h_n(x, \gamma_{V^c} \cup (\eta_2+u)) | + |h_n(x,(\eta_2+u)) |,
\]
if $\gamma \cup \eta_1 \cup (\eta_2+u) \in A_{n,0}(C)$, our choice of $C$ implies that
\[
	\sup_{x\in\Delta}|h_n(x,\gamma_{V^c})| \leq C(1 + k +l).
\]
Similarly if  $\gamma \cup \eta_1 \cup (\eta_2+u) \in A_{n,\Delta+u}(C)$, we have
\[
	\sup_{x\in \Delta +u} \left|h_n(x,\gamma_{V^c})\right| \leq C(1+k+l).
\]
It remains to control the energy of the configuration $\eta_1 \cup (\eta_2 + u)$, which can be written as 
\[
	H_n(\eta_1 \cup(\eta_2 + u)) = H_n(\eta_1) + H_n(\eta_2+u) + \sum_{x_1\in\eta_1}\sum_{x_2\in\eta_2} g_n(x_1 - x_2 - u).
\]
Since $\eta_2 + u \subset \Lambda_n$ we have $\eta_2 + u =\tau_u^n(\eta_2)$ and then $H_n(\eta_2+u) = H_n(\eta_2)$. The choice of the constant $C$ also implies that
\[
	\left|\sum_{x_1\in\eta_1}\sum_{x_2\in\eta_2} g_n(x_1 - x_2 - u)\right|\leq  Ckl.
\] 
Finally we obtain that if $\gamma_{V^c} \cup \eta_1 \cup \eta_2 \in \mathcal{A}_{n,u,\varepsilon}$, then 
\[
	\left|H_{n,V}(\eta_1 \cup (\eta_2+u), \gamma) - H_n(\eta_1) + H_n(\eta_2)\right|\leq 2C(k+1)(l+1),
\]
and similarly   
\[
	\left|H_{n,V}((\eta_1+u) \cup \eta_2), \gamma) - H_n(\eta_1) + H_n(\eta_2)\right|\leq 2C(k+1)(l+1).
\]
\end{proof}

\begin{proof}[Proof of Lemma \ref{L:control.proba.Pbn}]
From the assumption $\Delta + u \subset \Lambda_n$ we have $\Delta + u = \tau_u^n(\Delta)$ and then $\gamma \in A_{n,\Delta+u}(C)$ implies that $\tau^n_u(\gamma) \in A_{n,\Delta}(C)$. Since $\Pbn$ is invariant under the torus translations, $\Pbn(A_{n,\Delta + u}(C)) = \Pbn(A_{n,\Delta}(C))$. Therefore, we only have to show the result for $u = 0$ and the constant $C$ depends only on $\varepsilon$ and $\Delta$. We fix a reference point $x_0 \in \Delta$ and for $x\in\Delta$ we write  
\[
	\vert h_n(x, \gamma_{\Delta^c}) \vert \leq \vert h_n(x_0, \gamma_{\Delta^c}) \vert + \vert h_n(x, \gamma_{\Delta^c}) - h_n(x_0, \gamma_{\Delta^c}) \vert 
\]

and control separately both terms.

\textit{(1)} Under the probability measure $\Pbn$, the expectation of $\vert h_n(0, \cdot) \vert$ is bounded by a constant independent of $n$. So, applying Markov's inequality, for $C > 0$ large enough $\Pbn( \vert h_n(x_0, \cdot)\vert \leq C) \geq 1 - \varepsilon$ for every $n$. The local energy according to the exterior of $\Delta$ can be upper-bounded
\[
	\vert h_n(x_0, \gamma_{\Delta^c}) \vert \leq \vert h_n(x_0, \gamma) \vert + \sum_{y \in \gamma_\Delta} \vert g_n(y-x_0) \vert.
\] 	
According to Lemma \ref{L:majoration.n}, if $n$ is large enough, $\vert g_n(y-x_0) \vert \leq C/2 + g(y-x_0)$. Then we choose a radius $r>0$ small enough to have $\Pbn(N_{B(x_0, r)} = 0) \geq 1 - \varepsilon$ for every $n$ (such a choice is possible since for every $n$, the point process $\Pbn$ has intensity equal to $1$). For that choice $\vert g(y-x_0) \vert \leq r^{-s}$. In conclusion, for $C>0$ large enough and for every $n$ large enough as well
\[
	\Pbn\left(\big\{\gamma\in\Conf_{\Lambda_n} : \left| h_n(x_0,\gamma_{\Delta^c})\right| \leq C(1 + N_\Delta(\gamma)) \big\}\right) \geq 1 - \varepsilon.
\]
\textit{(2)} According to Lemma \ref{L:majoration.n}, there exists a constant $C_1$ (depending on $\Delta$ and $x_0$) such that for every $n$ large enough 
\[
	\left| g_n(y-x) - g_n(y-x_0) \right| \leq \left| g(y-x) - g(y-x_0) \right| + C_1 n^{-(s+1)/d}.
\]
It implies that for $\gamma \in \Conf_{\Lambda_n}$ such that $\vert \gamma \vert = n$
\begin{equation}\label{E:control.proba.Pbn.1}
	\left| h_n(x,\gamma_{\Delta^c}) - h_n(x_0,\gamma_{\Delta^c})\right| \leq C_1 +  \sum_{y \in \gamma_{\Delta^c}} \left|g(y-x) - g(y-x_0) \right|.
\end{equation}
To avoid points of $\gamma_{\Delta^c}$ close to $\Delta$, we introduce for $r > 0$, $\Delta_r = \left\{ x : d(x,\Delta) \leq r\right\}$. Since $\Pbn$ has intensity equal to $1$ for all $n\ge 1$, by Markov's inequality, $r$ can be chosen small enough in order to have for all $n\ge 1$, $\Pbn(N_{\Delta_r \setminus \Delta} = 0) \geq 1 - \varepsilon$. If we assume that $N_{\Delta_r \setminus \Delta}(\gamma) =0$
 and if $p_0$ is such that $\Delta_r \subset \Lambda_{p_0}$, then 
\begin{equation}\label{E:control.proba.Pbn.2}
	\sum_{y\in\gamma_{\Delta^c}} \vert g(y-x) - g(y-x_0)\vert \leq \sum_{p = p_0}^{+\infty} \sum_{y \in \gamma_{\Lambda_{p+1}\setminus \Lambda_p}} \vert g(y-x) - g(y-x_0)\vert + 2 N_{\Lambda_{p_0}\setminus \Delta}(\gamma) r^{-s}.
\end{equation}
There exists a constant $C_2>0$ (depending on $\Delta$) such that for $p_0$ large enough, if $y \in \Lambda_p^c$ with $p\geq p_0$, then $\vert g(y-x) - g(y-x_0) \vert \leq C_2 p^{-(s+1)}$. Combined with (\ref{E:control.proba.Pbn.1}) and (\ref{E:control.proba.Pbn.2}) it gives for $\gamma$ satisfying $N_{\Delta_r \setminus \Delta}(\gamma) =0$
\[
	\sup_{x\in\Delta} \left|h_n(x,\gamma_{\Delta^c}) - h_n(x_0,\gamma_{\Delta^c}) \right| \leq C_1 + C_2 \sum_{p = p_0}^{+\infty} N_{\Lambda_{p+1} \setminus \Lambda_p}(\gamma) p^{-(s+1)} + N_{\Lambda_{p_0}\setminus \Delta}(\gamma) r^{-s}.
\]
The right hand side has a finite expectation under $\Pbn$ with a value independent of $n$. By Markov's inequality, there exists a constant $C > 0$ such that for every $n$ 
\[
	\Pbn\left(C_1 + C_2 \sum_{p = p_0}^{+\infty} N_{\Lambda_{p+1} \setminus \Lambda_p} p^{-(s+1)} + N_{\Lambda_{p_0}\setminus \Delta} r^{-s} \leq C\right) \leq 1 -\varepsilon.
\]  
 According to our choice of $r$ and $p_0$, for such a constant $C$, we have for every $n$ large enough 
\[
	\Pbn\left(\left\{\gamma\in\Conf_{\Lambda_n} : \sup_{x\in\Delta}\left| h_n(x,\gamma_{\Delta^c}) - h_n(x_0,\gamma_{\Delta^c})\right| \leq C \right\}\right) \geq 1 - 2\varepsilon.
\]
\end{proof}

\subsubsection{Ergodic decomposition of $\Pbstar$ and consequences}
Let $\I$ be the $\sigma$-algebra of events invariant by translation. Since $\Pbstar$ is stationary, for every event $A$ and $\Pb$-almost every $\gamma$ we have
\begin{equation}\label{E:theorem.ergo}
	\lim_{M\rightarrow +\infty} \frac{1}{M} \sum_{m=1}^M 1_{A}(\gamma + mu) = \Pbstar(A \mid \I)(\gamma).
\end{equation}
There exists a version of the conditional expectation $A \mapsto E_{\Pbstar}(1_A \mid \I)$ such that for $\Pb$-almost every $\gamma$, $A \mapsto \Pb_{\star,\gamma}(A) := E_{\Pbstar}(1_A \mid \I)(\gamma)$ is a probability measure on $\Conf$.

\begin{lem}\label{L:ergo.properties} For $\Pbstar$-almost every $\gamma$, the probability measure $\Pb_{\star,\gamma}$ is stationary, ergodic and satisfies the canonical DLR equations.
\end{lem}
\begin{proof}
The stationarity comes from \eqref{E:theorem.ergo}, and since for $\Pbstar$-almost every $\gamma$ and every $A\in\I$
\[
\Pb_{\star,\gamma}(A) =E_{\Pbstar}(1_A)(\gamma) = 1_A(\gamma) \in \{0,1\},
\]
the ergodicity is also proved. To obtain the canonical DLR equations, we recall the notation $\mathcal{E}_\Delta = \sigma(N_\Delta, \F_{\Delta^c})$ and introduce 
\[
	\mathcal{E}_\infty = \bigcap_{m\geq 1} \mathcal{E}_{\Lambda_m}.
\]
Let $f$ be a bounded local function and $\Delta$ a bounded Borel subset of $\RRd$, Remark~\ref{r.esperance.cond} implies that 
\begin{equation}\label{E:DLR.asympt}
	E_{\Pbstar}(f\mid \mathcal{E}_\infty) =E_{\Pbstar}(E_{\Pbstar}(f \mid \mathcal{E}_\Delta) \mid \mathcal{E}_\infty) = E_{\Pbstar}(f_\Delta \mid \mathcal{E}_\infty).
\end{equation}
According to \eqref{E:theorem.ergo}, $E_{\Pbstar}(f \mid \I)$ is $\mathcal{E}_\infty$-mesurable and is equal to $E_{\Pbstar}(f \mid \I \cap \mathcal{E}_\infty)$, then for $\Pbstar$-almost every $\gamma$, $\Pb_{\star,\gamma}(\cdot) = E_{\Pbstar}(1_{\cdot} \mid \I \cap \mathcal{E}_\infty)(\gamma)$. Combined with \eqref{E:DLR.asympt}, it gives successively 
\[
	E_{\Pb_{\star,\gamma}}(f) =E_{\Pbstar}(f \mid \I \cap \mathcal{E}_\infty) = E_{\Pbstar}(E_{\Pbstar}(f \mid \mathcal{E}_\infty) \mid \I \cap \mathcal{E}_\infty)= E_{\Pbstar}(f_\Delta \mid \I \cap \mathcal{E}_\infty) = E_{\Pb_{\star,\gamma}}(f_\Delta)
\]
which are the canonical DLR equations for $\Pb_{\star,\gamma}$.
\end{proof}

\begin{prop}\label{C:ergodic.proba}
Let $\Delta$ be a bounded Borel subset with $\leb(\Delta)>0$ and $k \geq 0$ an integer. For $\Pbstar$-almost every $\gamma$ we have 
\[
	\Pbstar(N_\Delta = k \mid \I)(\gamma) > 0.
\] 
\end{prop}
\begin{proof}
This is an adaptation of Corollary 3.3 of \cite{DLRsinebeta} which shows that since $\Pb_{\star,\gamma}$ satisfies the canonical DLR equations then for every integer $k$, $\Pb_{\star,\gamma}(N_\Delta  = k)> 0$. 
\end{proof}

\subsubsection{Proof of Theorem \ref{T:point.deletion}}
The next proposition gives a generalisation of Proposition \ref{P:inequality.Pbn} for the infinite volume measure $\Pbstar$.
\begin{prop}\label{P:inequality.Pb}
For every compact set $\Delta$ with $\leb(\Delta)>0$, every pair of integers $(k,l)$, every $\varepsilon > 0$ and every $\F_{\Delta^c}$-measurable bounded function $f$ we have 
\begin{multline*}
\int \Pbstar(N_\Delta = k \mid \F_{\Delta^c})(\gamma) \Pbstar(N_\Delta = l \mid \mathcal{I})(\gamma) f(\gamma) \Pbstar(d\gamma) \\ \leq C^{k,l}_{\Delta, \varepsilon} \int \Pbstar(N_\Delta = l \mid \F_{\Delta^c})(\gamma) \Pbstar(N_\Delta = k \mid \mathcal{I})(\gamma) f(\gamma) \Pbstar(d\gamma)  + \varepsilon \normf.
\end{multline*}
\end{prop}
Before proving this proposition, we show how to derive Theorem \ref{T:point.deletion} from it. Given a compact set $\Delta$ with $\leb(\Delta)>0$ and two integers $k,l \geq 0$, we introduce the event 
\[
	A_{k,l} = \left\{\gamma \in \Conf : \Pbstar(N_\Delta = k \mid \F_{\Delta^c})(\gamma) > 0,\; \Pbstar(N_\Delta = l \mid \F_{\Delta^c})(\gamma) = 0 \right\}.
\]
Roughly speaking, it is the set of configurations for which, given the outside of $\Delta$, it is possible to have $k$ points in $\Delta$, but impossible to have $l$ points. Applying Proposition~\ref{P:inequality.Pb} to the $\F_{\Delta^c}$-measurable function $1_{A_{k,l}}$, gives for every $\varepsilon > 0$ the inequality  
\[
\int \Pbstar(N_\Delta = k \mid \F_{\Delta^c})(\gamma) \Pbstar(N_\Delta = l \mid \mathcal{I})(\gamma) 1_{A_{k,l}}(\gamma) \Pbstar(d\gamma) \leq \varepsilon .
\]
It implies that for $\Pbstar$-almost every $\gamma$,
\[
	\Pbstar(N_\Delta = k \mid \F_{\Delta^c})(\gamma) \Pbstar(N_\Delta = l \mid \mathcal{I})(\gamma) 1_{A_{k,l}}(\gamma) = 0.
\]	
By Proposition \ref{C:ergodic.proba} and by definition of the event $A_{k,l}$ we found that $\Pbstar(A_{k,l}) = 0$ which shows that for $\Pbstar$-almost every $\gamma$,
\[
	\Pbstar(N_\Delta = k \mid \F_{\Delta^c})(\gamma) > 0 \implies \Pbstar(N_\Delta = l\mid \F_{\Delta^c})(\gamma) > 0.
\]
Since for $\Pbstar$-almost every $\gamma$ there exists $k_0$ such that $\Pbstar(N_\Delta = k_0 \mid \F_{\Delta^c})(\gamma) > 0$, the Theorem \ref{T:point.deletion} is proved. 
\begin{proof}[Proof of Proposition \ref{P:inequality.Pb}] According to Proposition \ref{P:inequality.Pbn}, the subsequence $(\Pbnj)_{j\geq 1}$, converging to $\Pbstar$, satisfies for $j$  large enough 
\[
	\int 1_{N_\Delta = k}(\gamma) 1_{N_{\Delta + u} = l}(\gamma) f(\gamma) \Pbnj(d\gamma) \leq C^{k,l}_{\Delta, \varepsilon} \int 1_{N_\Delta = l}(\gamma) 1_{N_{\Delta+u} = k}(\gamma) f(\gamma) \Pbnj(d\gamma) + \varepsilon \normf.
\]
where $f$ is a $\F_{\Lambda\setminus \Delta}$-measurable bounded function and $u$ satisfying $d(\Lambda, \Lambda + u) > 1$. Using the local convergence to $\Pb_\star$, that leads to
\[
	\int 1_{N_\Delta = k}(\gamma) 1_{N_{\Delta + u} = l}(\gamma) f(\gamma) \Pbstar(d\gamma) \leq C^{k,l}_{\Delta, \varepsilon} \int 1_{N_\Delta = l}(\gamma) 1_{N_{\Delta+u} = k}(\gamma) f(\gamma) \Pbstar(d\gamma) + \varepsilon \normf.
\]
Since $f$ is a $\F_{\Lambda\setminus\Delta}$-measurable function and $1_{N_{\Delta + u}=k}$ is $\F_{\Delta^c}$-measurable, by definition of the conditional expectation we have
\begin{multline*}
	\int \Pbstar(N_\Delta = k \mid \F_{\Delta^c})(\gamma) 1_{N_{\Delta + u} = l}(\gamma) f(\gamma) \Pbstar(d\gamma) \\ \leq C^{k,l}_{\Delta, \varepsilon} \int \Pbstar(N_\Delta = l \mid \F_{\Delta^c})(\gamma) 1_{N_{\Delta+u} = k}(\gamma) f(\gamma) \Pbstar(d\gamma) + \varepsilon \normf.
\end{multline*}
We choose $u$ such that for all $m \geq 1$, $d(\Lambda, \Delta + mu) > 1$. Since the constant $C^{k,l}_{\Delta,\varepsilon}$ does not depend on $u$, the previous inequality is also true for $mu$. By summing over $m = 1, ..., M$ and dividing by $M$ we obtain 
\begin{multline*}
	\int \Pbstar(N_\Delta = k \mid \F_{\Delta^c})(\gamma) \left(\frac 1M\sum_{m=1}^M 1_{N_{\Delta + mu} = l}(\gamma)\right) f(\gamma) \Pbstar(d\gamma) \\ \leq C^{k,l}_{\Delta, \varepsilon} \int \Pbstar(N_\Delta = l \mid \F_{\Delta^c})(\gamma) \left( \frac 1M \sum_{m=1}^M 1_{N_{\Delta+mu} = k}(\gamma) \right) f(\gamma) \Pbstar(d\gamma) + \varepsilon \normf.
\end{multline*}
From the ergodic Theorem and the dominated convergence Theorem, letting $M$ goes to infinity, we deduce that 
\begin{multline*}
	\int \Pbstar(N_\Delta = k \mid \F_{\Delta^c})(\gamma) \Pbstar(N_\Delta = l\mid \mathcal{I})(\gamma) f(\gamma) \Pbstar(d\gamma) \\ \leq C^{k,l}_{\Delta, \varepsilon} \int \Pbstar(N_\Delta = l \mid \F_{\Delta^c})(\gamma)\Pbstar(N_\Delta = k\mid \mathcal{I})(\gamma) f(\gamma) \Pbstar(d\gamma) + \varepsilon \normf.
\end{multline*}
The previous inequality is true for every bounded $\F_{\Lambda\setminus\Delta}$-measurable function. If $f$ is a $\F_{\Delta^c}$-measurable function, we define $f_p = E_{\Pb}(f \mid \F_{\Lambda_p\setminus \Delta})$. Since $\Vert f_p \Vert_{\infty}^{} \leq \normf$, we have for $p$ large enough 
\begin{multline*}
	\int \Pbstar(N_\Delta = k \mid \F_{\Delta^c})(\gamma) \Pbstar(N_\Delta = l\mid \mathcal{I})(\gamma) f_p(\gamma) \Pbstar(d\gamma) \\ \leq C^{k,l}_{\Delta, \varepsilon} \int \Pbstar(N_\Delta = l \mid \F_{\Delta^c})(\gamma)\Pbstar(N_\Delta = k\mid \mathcal{I})(\gamma) f_p(\gamma) \Pbstar(d\gamma) + \varepsilon \normf.
\end{multline*}
The martingale $(f_p)_{p\geq 1}$ converges $\Pb$-almost surely, when $p$ goes to infinity, to \(\E \Pb (f \mid \F_{\Delta^c})\), which is equal to $f$. From the dominated convergence theorem and passing to the limit in the inequalities, we conclude the proof. 
\end{proof}

\subsection{Grand canonical DLR equations}
The proofs of Proposition \ref{P:local.energy} and Theorem \ref{T:grand.canonique} rely on properties of the \emph{Campbell measures} of the point process $\Pbstar$. 
\begin{defin}\label{D:campbell}
Let $\P$ be a point process and $n$ a positive integer, the  Campbell measure of order $n$ of $\P$ is the measure $\Camp^{(n)}_\P$ on $(\RRd)^n\times \Omega$ defined for any non-negative test function $f : (\RRd)^n \times \Conf \to \RR$ as follows
\[
	\Camp^{(n)}_{\P}(f) =\int \sum_{\substack{x_1, ..., x_n \in\gamma \\ \neq}} f(x_1, ..., x_n ,\gamma \setminus \{x_1, ..., x_n\}) \P(d\gamma),
\]
where the summation is performed  over all $n$-tuples of distinct points in $\gamma$.
\end{defin}
Using Theorem \ref{T:point.deletion}, we show that the Campbell measure of order one has a density with respect to $\leb \otimes \Pbstar$. Then we obtain a description of Campbell measures for other order and finally we construct the compensator mentioned in Proposition \ref{P:local.energy}.  
\subsubsection{Structure of the one point Campbell measure}

For $x\in\RRd$ and $\gamma$ a point configuration, the cost of moving a point from the origin to the position $x$ in the field created by $\gamma$ is defined by
\[
	V(x,\gamma) = \lim_{p\rightarrow +\infty} \sum_{y\in \gamma_{\Lambda_p}} \left[g(y-x) - g(y)\right].
\]
According to Lemma \ref{L:move.existence}, if $\P$ is a point process with bounded intensity, then for $\P$-almost every $\gamma$, the previous limit exists and is finite. This cost function plays a crucial role in the structure of the Campbell measures for canonical Gibbs point processes. 
We start by the properties of the Campbell measure of order one. 
\begin{prop}
There exists a measure $\Qbstar$ on $\Conf$ such that $\Camp^{(1)}_{\Pbstar}$ is absolutely continuous with respect to $\lambda^d \otimes \Qbstar$ and has density
\[
	\frac{d\Camp_{\Pbstar}}{d\lambda^d \otimes \Qbstar}(x,\gamma) = e^{-\beta V(x,\gamma)}.
\]
\end{prop}
\begin{proof}
It is a consequence of the canonical DLR equations, and the proof is the same as the one of Theorem 3.7 in \cite{DLRsinebeta}. 
\end{proof}
The previous result is valid has soon as the point process satisfies the canonical DLR equations, and in particular, it is true for every $\beta$-Circular Riesz gas. The next result is specific to the point process $\Pbstar$, and rely on the property highlighted in Theorem \ref{T:point.deletion}. 

\begin{thm}\label{T:campbell.equiv}
The measures $\Pbstar$ and $\Qbstar$ are equivalent.\end{thm}
\begin{proof}
Let $\Delta\subset\RRd$ be a compact set with $\leb(\Delta)>0$ and $E \in \F$ an event. Introducing the number of points in $\Delta$ we have
\[
\Camp_{\Pbstar}(\Delta \times E) = \int \sum_{x\in\gamma} 1_{\Delta}(x) 1_{E}(\gamma\setminus x) \Pbstar(d\gamma) = \sum_{k=1}^{+\infty}\int \sum_{x\in\gamma} 1_{\Delta}(x) 1_{E}(\gamma\setminus x) 1_{N_\Delta = k}(\gamma) \Pbstar(d\gamma).
\]
The canonical DLR equations lead to 
\[
\Camp_{\Pbstar}(\Delta \times E) = \sum_{k=1}^{+\infty}\iint \sum_{x\in\eta} 1_{E}(\gamma_{\Delta^c}\cup\eta\setminus x) 1_{N_\Delta = k}(\gamma)G_\Delta(\eta,\gamma_{\Delta^c},k) \Bin {\Delta} k(d\eta) \Pbstar(d\gamma),
\]
where 
\[
	G_\Delta(\eta, \gamma,k) = \frac{1}{Z_\Delta(\gamma,k)}\exp(-\beta(H(\eta) + M_\Delta(\eta,\gamma)))1_{N_\Delta(\eta) = k},
\]
with the partition function
\[
	Z_\Delta(\gamma,k) = \int \exp(-\beta(H(\eta) + M_\Delta(\eta,\gamma))) \Bin {\Delta} k(d\eta).
\]
These expressions depend only on $\gamma_{\Delta^c}$ (i.e. they are $\F_\Delta$-measurable with respect to $\gamma$). Using Lemma~\ref{L:GNZ.finite.volume} for the binomial process ($\beta = 0$), 
\[
\Camp_{\Pbstar}(\Delta \times E)  = \sum_{k=1}^{+\infty} \iiint 1_{E}(\gamma_{\Delta^c}\cup\eta) 1_{N_\Delta = k}(\gamma)G_\Delta(\eta\cup x,\gamma_{\Delta^c},k) \lambda^d_\Delta(dx)\Bin {\Delta} {k-1}(d\eta) \Pbstar(d\gamma).
\]
According to the properties of the conditional expectation and using Theorem \ref{T:point.deletion}, we obtain
\begin{multline*}
\Camp_{\Pbstar}(\Delta \times E)  = \sum_{k=1}^{+\infty} \iiint 1_{E}(\gamma_{\Delta^c}\cup\eta) \frac{\Pbstar(N_\Delta = k \mid \F_{\Delta^c})(\gamma)}{\Pbstar(N_\Delta = k-1 \mid \F_{\Delta^c})(\gamma)}G_\Delta(\eta\cup x,\gamma_{\Delta^c},k) \\ 1_{N_\Delta = k-1}(\gamma)\lambda^d_\Delta(dx)\Bin {\Delta}{k-1}(d\eta) \Pbstar(d\gamma).
\end{multline*}
Since for $\Bin {\Delta} {k-1}$-almost every $\eta$ and $\Pbstar$-almost every $\gamma$, $G_\Delta(\eta,\gamma,k-1) > 0$ and so it can be introduced in the integral
\begin{multline*}
\Camp_{\Pbstar}(\Delta \times E)  = \sum_{k=1}^{+\infty} \iiint 1_{E}(\gamma_{\Delta^c}\cup\eta) \frac{\Pbstar(N_\Delta = k \mid \F_{\Delta^c})(\gamma)}{\Pbstar(N_\Delta = k-1 \mid \F_{\Delta^c})(\gamma)}\frac{G_\Delta(\eta\cup x,\gamma,k)}{G_\Delta(\eta,\gamma,k-1)} \\ 1_{N_\Delta = k-1}(\gamma) G_\Delta(\eta,\gamma,k-1)\lambda^d_\Delta(dx)\Bin {\Delta} {k-1}(d\eta) \Pbstar(d\gamma).
\end{multline*}
Using again the canonical DLR equations,
\begin{multline*}
\Camp_{\Pbstar}(\Delta \times E)  = \sum_{k=1}^{+\infty} \int \frac{\Pbstar(N_\Delta = k \mid \F_{\Delta^c})(\gamma)}{\Pbstar(N_\Delta = k-1 \mid \F_{\Delta^c})(\gamma)}\left(\int_\Delta \frac{G_\Delta(\gamma\cup x,\gamma,k)}{G_\Delta(\gamma,\gamma,k-1)}dx\right)  \\ 1_{E}(\gamma) 1_{N_\Delta = k-1}(\gamma) \Pbstar(d\gamma).
\end{multline*}
We deduce from the previous equation that if $\Pbstar(E) = 0$ then $\Camp_{\Pbstar}(\Delta\times E) = 0$, since it holds for every compact set $\Delta$, we have $\Qbstar(E) = 0$. Observing that for $\Pbstar$-almost every~$\gamma$ the function 
\[
\gamma \mapsto \sum_{k=1}^{+\infty}  \frac{\Pbstar(N_\Delta = k \mid \F_{\Delta^c})(\gamma)}{\Pbstar(N_\Delta = k-1 \mid \F_{\Delta^c})(\gamma)}\left(\int_\Delta \frac{G_\Delta(\gamma\cup x,\gamma,k)}{G_\Delta(\gamma,\gamma,k-1)}dx\right)  1_{N_\Delta = k-1}(\gamma)
\]
is strictly positive, we also deduce that $\Qbstar(E) = 0$ implies $\Pbstar(E) = 0$. At the end, that proves that $\Qbstar$ and $\Pbstar$ are equivalent.
\end{proof}

In particular, the previous result ensures the existence of a function $\Psibstar : \Conf \rightarrow \RR$ such that 
\begin{equation}\label{E:Psibstar}
\frac{d\Qbstar}{d\Pbstar}(\gamma) = \exp\big(-\beta \Psibstar(\gamma) \big).
\end{equation}

\subsubsection{Structure of the $n$-point Campbell measure}
The fonction $\Psibstar$ introduced in~\eqref{E:Psibstar} can be used to describe the Campbell measure of order $n$ for every $n\geq 1$. 

\begin{prop}\label{P:n.point.Campbell}
For every integer $n\geq 1$ and every positive measurable function $f$
\begin{multline*}
\Camp^{(n)}_{\Pb_\star}(f) = \int f(x_1, ..., x_n, \gamma) \exp\left\{ - \beta H(\{x_1, ..., x_n\}) - \beta\sum_{j=1}^n V(x_j,\gamma) \right.\\\left.- \beta\sum_{j=0}^{n-1} \left[\Psibstar(\gamma\cup\{x_1, ..., x_j \}) - \sum_{i=1}^j g(x_i)\right]\right\}dx_1 ... dx_n\, \Pbstar(d\gamma).   
\end{multline*}
\end{prop}
\begin{proof} We prove the result by induction on $n$. For $n =1$ the previous expression becomes 
\[
	\Camp_{\Pbstar}^{(1)}(f) = \iint f(x,\gamma) \exp\left(-\beta (V(x,\gamma) + \Psibstar(\gamma))\right) \Pbstar(d\gamma) dx,
\]
which is true from the definition of $\Psibstar$ introduced in \eqref{E:Psibstar}. Let $n\geq 1$ be given, and assume that the result of Proposition~\ref{P:n.point.Campbell} is true for $n$. The Campbell measure of order $(n+1)$ can be written as  
\[
\Camp^{(n+1)}_{\Pbstar}(f) = \int \sum_{x_{n+1} \in \gamma} \,\sum_{\substack{x_1, ..., x_n \in \gamma\setminus x_{n+1}\\ \neq }} f(x_1, ..., x_{n+1}, \gamma \setminus \{x_1, ..., x_{n+1}\})\Pbstar(d\gamma).
\]
The result for the Campbell measure of order one gives
\[
\Camp^{(n+1)}_{\Pbstar}(f) = \iint \sum_{\substack{\{x_1, ..., x_n\} \subset \gamma\\ \neq}} f(x_1, ..., x_{n+1}, \gamma \setminus \{x_1, ..., x_n\})e^{-\beta(V(x_{n+1}, \gamma) + \Psibstar(\gamma))} \Pbstar(d\gamma)dx_{n+1}.
\]
Applying the result for the Campbell measure of order $n$ and use the identity 
\[
	H(\{x_1, ..., x_n\})+ V(x_{n+1}, \gamma \cup\{x_1, ..., x_n\}) = H(\{x_1, ..., x_{n+1}\}) + V(x_{n+1},\gamma) -\sum_{j=1}^n g(x_j),
\]
that leads to the expected expression for the Campbell measure of order $(n+1)$.
\end{proof}

It is interesting to point out the three terms appearing in the exponential in Proposition~\ref{P:n.point.Campbell}. The term $H(\{x_1, \dots, x_n\})$ is the energy of the interaction between the points $x_1, \dots, x_n$ and $\sum_{j=1}^n V(x_j,\gamma)$ is the cost for moving the points from the origin to their positions in the field created by $\gamma$. The third one is a little bit more complicated and partially analysed in the next corollary. 
\begin{cor}\label{C:eq.Psi}
For every integer $n$, for $\Pbstar$-almost every $\gamma$ and $(\lambda^d)^{\otimes k}$-almost every $x_1, \dots, x_k$ and $y_1, \dots, y_k$ 
\[
	\Psibstar(\gamma\cup\{x_1, \dots, x_n \}) - \sum_{j=1}^n g(x_j) = 
	\Psibstar(\gamma\cup\{y_1, \dots, y_n \}) - \sum_{j=1}^n g(y_j).
\]

\end{cor}
\begin{proof}
It is obvious for $n=0$. For $n\geq 1$, let $f : \RR^{n+1}\times \Conf \rightarrow \RR$ be a test function, and introduce the function $g$ defined by $g(x_1, ..., x_{n+1}, \gamma) = f(x_1, ..., x_{n+1}, x_{n}, \gamma)$. According to the definition of the Campbell measure, $\Camp^{(n)}_{\Pbstar}(f) = \Camp^{(n)}_{\Pbstar}(g)$ and comparing the expressions obtained with Proposition~\ref{P:n.point.Campbell}, we deduce that for $\Pbstar$-almost every $\gamma$ and $\lambda^d$-almost every $x_1, \dots, x_{n}, y_{n} \in \RRd$ we have
\[
	\Psibstar(\gamma\cup\{x_1, ..., x_{n-1}, x_{n} \}) - \sum_{j=1}^{n} g(x_j) = 
	\Psibstar(\gamma\cup\{x_1, ..., x_{n-1}, y_{n} \}) - \sum_{j=1}^{n-1} g(x_j) - g(y_n).
\]
Applying that property $(n-1)$ times, that gives the result. 
  
\end{proof}
The quantities $\Psibstar(\gamma\cup\{x_1, ..., x_j \}) - \sum_{i=1}^j g(x_i)$ appearing in Proposition~\ref{P:n.point.Campbell} do not depend on the points $x_1, \dots, x_n$ and can be interpreted as the cost of adding a point at the origin in  the configuration $\gamma$ when $j$ points has been previously added to $\gamma$. Then the sum $\sum_{j=1}^{n}\left[\Psibstar(\gamma\cup\{x_1, ..., x_j \}) - \sum_{i=1}^j g(x_i)\right]$ is the cost of adding $n$ points at the origin. This sum is the equivalent, in our case, of the function $\textsf{create}_n$ introduced in Theorem 3.15 of \cite{DLRsinebeta}. 

In particular, for $n=1$, $\Psibstar(\gamma)$ corresponds to the cost of adding one point at the origin. We must be careful, nothing ensures that the cost of adding two points is twice the cost of adding one point, and we cannot say that the previous sum is equal to $n \Psibstar(\gamma)$. It would be true if we had for instance $\Psibstar(\gamma\cup\{x\}) = \Psibstar(\gamma) + g(x)$.
\subsubsection{Construction of the compensator}\label{S:compensator.story}
We start by describing the effect on $\Psibstar$ of simultaneously moving $k$~points in a point configuration. 
\begin{prop}\label{P:eq.Psi}
For all integers $m,n$ and $k$, for $\Pbstar$-almost every $\gamma$, for $\lambda^d$-almost every $y_1, ..., y_{n+k}$ and every $\{x_1, ..., x_{m+k}\} \subset \gamma$ we have 
\begin{multline*}
	\Psibstar(\gamma \setminus \{x_1, ..., x_{m+k}\} \cup \{y_1, ..., y_{n+k}\})- \sum_{j=1}^k g(y_{n+j})  		\\
	= \Psibstar(\gamma\setminus\{x_1, \dots, x_m\} \cup\{y_1, \dots, y_n\}) 	
	-\sum_{j=1}^k  g(x_{m+j}).
\end{multline*}
\end{prop}
\begin{proof}
We denote the vectors $(x_1, \dots, x_{m+k})$ and $(y_1, \dots, y_{n+k})$ respectively by $X_{m+k}$ and $Y_{n+k}$. Let $h : \mathbb{R}^{m+k} \times \mathbb{R}^{n+k} \times \Conf \to [0,+\infty)$ be a test function and, for $Y_{n+k} \in \mathbb{R}^{n+k}$, we introduce the partial function 
\[
h^{(1)}_{Y_{n+k}} : (X_{m+k},\gamma) \mapsto h(X_{m+k},Y_{n+k}, \gamma) e^{-\Psi(\gamma\cup\{y_1, ..., y_n, x_{m+1}, \dots, x_{m+k}\})}.
\]
By definition of the Campbell measure of order $(m+k)$ 
\begin{equation}\label{E:campbell.1}
	\Camp^{(m+k)}_{\Pbstar}(h^{(1)}_{Y_{n+k}}) = \int \sum_{\substack{\{x_1, \dots, x_{m+k}\}\subset\gamma\\ \neq}} h(X_{m+k},Y_{n+k}, \gamma) 	e^{-\Psi(\gamma\setminus\{x_1,\dots,x_m\}\cup\{y_1, \dots y_n\})}
	 \Pbstar(d\gamma) dY_{n+k}. 
\end{equation}
But according to Corollary~\ref{C:eq.Psi} and Proposition~\ref{P:n.point.Campbell}, for $\lebn{(n+k)}$-almost every $Y_{n+k}$ 
\[
\Camp^{(m+k)}_{\Pbstar}(h^{(1)}_{Y_{n+k}}) = \Camp^{(m+k)}_{\Pbstar}(h^{(2)}_{Y_{n+k}})
\]
where we have introduced a second partial function
\[
h^{(2)}_{Y_{n+k}} : (X_{m+k},\gamma) \mapsto h(X_{m+k},Y_{n+k}, \gamma) e^{-\beta\Psibstar(\gamma\cup\{y_1, ..., y_{n+k}\})+\beta \sum_{j=1}^{k} g(y_{n+j})-g(x_{m+j})}.	
\]
Again, if we use the definition of the Campbell measure 
\begin{multline}\label{E:campbell.2}
	\Camp^{(m+k)}_{\Pbstar}(h^{(2)}_{Y_{n+k}}) = \int \sum_{\substack{\{x_1, \dots, x_{m+k}\}\subset\gamma \\ \neq}} h(X_{m+k},Y_{n+k}, \gamma) 	e^{-\Psi(\gamma\setminus\{x_1,\dots,x_{m+k}\}\cup\{y_1, \dots y_{n+k}\})}\\
	e^{\beta \sum_{j=1}^{k} g(y_{n+j})-g(x_{m+j})}
	 \Pbstar(d\gamma) dY_{n+k}. 
\end{multline}
Since expressions \eqref{E:campbell.1} and \eqref{E:campbell.2} are equal for every test function $h$, the result is proved. 
\end{proof}

Let $\mathcal{N}_{m,n,k}$ be a negligible subset of $(\RRd)^n\times\Conf$ for $\lebn{n}\otimes\Pbstar$ such that the result of Proposition \ref{P:eq.Psi} is true for every $(y_1,\dots,y_n,\gamma) \in (\RRd)^n\times\Conf \setminus \mathcal{N}_{m,n,k}$ and $\leb$-almost every $y_{n+1}, \dots, y_{n+k}$. We introduce a set of \emph{admissible} point configurations
\begin{multline}\label{E:def.event.E}
E = \bigcup_{n,m\geq 0} \big\{\gamma\setminus\{x_1, \dots, x_m\} \cup\{y_1, \dots, y_n\}, (y_1, \dots, y_n, \gamma) \in \mathbb{R}^n \times \Conf\setminus \mathcal{N}_{m,n}, \\
\{x_1, \dots, x_m\} \subset \gamma \big\},
\end{multline}
where $\mathcal{N}_{m,n} = \bigcap_{k\geq 0}\mathcal{N}_{m,n,k}$. According to Proposition~\ref{P:eq.Psi}, if $\gamma \in E$ and $\Lambda$ is a bounded Borel set, then for almost every $y_1, \dots, y_{N_\Lambda(\gamma)} \in \RRd$ the quantity
\[
\Psibstar(\gamma_{\Lambda^c}\cup \set{y}{N_\Lambda(\gamma)}) - \sum_{j=1}^{N_{\Lambda}(\gamma)} g(y_j)
\]
only depends on $\gamma$ and $k$, and defines a function $\Cbstar$ on $\NN\times E$. We extend the function $\Cbstar$ to $\NN\times\Conf$ by taking it equals to $0$ on $\NN \times (\Conf\setminus E)$.  Moreover, if we use Proposition~\ref{P:eq.Psi} with $\gamma \in E$, $m=n=0$, $k=N_\Lambda(\gamma)$ and $\set{y}{k} = \gamma_\Lambda$ we have
\[
	\Psibstar(\gamma) = \sum_{x\in \gamma_\Lambda} g(x) + C^\beta_{\star}(N_\Lambda(\gamma), \gamma_{\Lambda}). 
\] 
In particular the following limit exists 
\[
	\Psibstar(\gamma) = \lim_{p\rightarrow +\infty} \sum_{x\in \gamma_{\Lambda_p}} g(x) + \Cbstar(N_{\Lambda_p}(\gamma), \gamma_{\Lambda_p}).
\]
If the local energy of a point is defined as $\hbstar(x,\gamma) = V(x,\gamma) + \Psibstar(\gamma)$ then for every $\gamma\in E$ we have 
\[
	\hbstar(x,\gamma) = \lim_{p\to +\infty} \sum_{y\in\gamma_{\Lambda_p}} g(x-y) + \Cbstar (N_{\Lambda_p}(\gamma), \gamma_{\Lambda_p}),
\]
which proves Proposition~\ref{P:local.energy}. We can remark that the compensator constructed here is in fact a  constant sequence of functions.  

Given a bounded Borel set $\Delta$ with $\leb(\Delta)>0$, the local energy of a configuration $\eta =\{x_1, \dots, x_n\}$ is defined as 
\[
	H^\beta_{\star,\Delta}(\eta, \gamma) = \hbstar(x_1, \gamma_{\Delta^c}) + \hbstar(x_2, \gamma_{\Delta^c}\cup\{x_1\}) + \dots + \hbstar(x_n,\gamma_{\Delta^c}\cup\{x_1, \dots x_{n-1}\}).
\]
Since for almost every $x_1, \dots, x_j \in \RRd$, $\gamma_{\Delta^c}\cup\set{x}{j} \in E$ (according to~\ref{E:def.event.E}) the local energy $\hbstar(x_j, \gamma_{\Delta^c}\cup\set{x}{j})$ is well defined. We also need to check that $H^\beta_{\star,\Delta}(\eta, \gamma)$ does not depend on the implicit order of the points of $\eta$ when we write $\eta =\{x_1, \dots, x_n\}$. This is the case since, according to Proposition~\ref{P:eq.Psi},  for $\gamma \in E$ and almost every $x,y \in \RR$ we have
\begin{align*}
h(x,\gamma) + h(y,\gamma\cup\{x\})& = V(x,\gamma) + \Psi(\gamma) + V(y, \gamma\cup\{x\}) + \Psi(\gamma \cup \{x\}) \\
& =  g(x-y) + V(x,\gamma) + \Psi(\gamma) + V(y, \gamma\cup) + \Psi(\gamma \cup \{x\}) -g(x) \\
 & = g(x-y) + V(x,\gamma) + \Psi(\gamma) + V(y, \gamma\cup) + \Psi(\gamma \cup \{y\}) -g(y) \\
& = V(y,\gamma) + \Psi(\gamma) + V(x, \gamma\cup\{y\}) + \Psi(\gamma \cup \{y\})\\
& = h(y,\gamma) + h(x,\gamma\cup\{y\}).
\end{align*}
To conclude, it remains to prove the grand canonical DLR equations. 
\begin{proof}[Proof of Theorem \ref{T:grand.canonique}]
According to Proposition~\ref{P:n.point.Campbell} and our previous definition of the local energy of a point,  we have 
\begin{equation}\label{E:GNZ}
	C^{(1)}_{\Pbstar} (f) = \iint f(x,\gamma) e^{-\beta \hbstar(x,\gamma)}  \Pbstar(d\gamma)dx.
\end{equation}
Equation \eqref{E:GNZ}, is called a Georgii-Nguyen-Zessin (GNZ) equation. Usually, DLR equations and GNZ equations are equivalent. To obtain that equivalence in our case, we follow the steps of the proof of Theorem 2 of \cite{dereudre2019introduction}. Let $\Delta$ be a bounded Borel set with $\leb(\Delta)>0$, we denote by $\gamma \mapsto \Pbstar(\cdot \mid \gamma_{\Delta^c})$ a regular version of the conditional expectation (then for $\Pbstar$-almost every $\gamma$ the application $A\in\F \mapsto \Pbstar(A\mid\gamma_{\Delta^c})$ is a probability measure). Applying \eqref{E:GNZ}  to the function 
$(x,\gamma) \mapsto g(\gamma_{\Delta^c})1_\Delta(x) f(x,\gamma_{\Delta})$ where $g$ is a positive measurable function and $f : \RRd \times \Conf \to [0,+\infty)$ a test function, leads to
\begin{multline*}
\int g(\gamma_{\Delta^c}) \int \sum_{x\in\eta_\Delta} f(x,\eta_{\Delta}\setminus \{x\}) \Pbstar(d\eta \mid \gamma_{\Delta^c}) \Pbstar(d\gamma)  \\
= \int g(\gamma_{\Delta^c}) \iint \sum_{x\in\eta_\Delta} f(x,\eta_{\Delta})e^{-\beta\hbstar(x,\eta_\Delta \cup \gamma_{\Delta^c})} \Pbstar(d\eta \mid \gamma_{\Delta^c})\leb_\Delta(dx) \Pbstar(d\gamma)
\end{multline*}
It implies that for $\Pbstar$-almost every $\gamma$, the probability measure $\Pbstar(\cdot \mid \gamma_{\Delta^c})$ satisfies the GNZ equation
\begin{multline*}
\int \sum_{x\in\eta_\Delta} f(x,\eta_{\Delta}\setminus \{x\}) \Pbstar(d\eta \mid \gamma_{\Delta^c}) \Pbstar(d\gamma)  
 \\
 = \iint \sum_{x\in\eta_\Delta} f(x,\eta_{\Delta})e^{-\beta\hbstar(x,\eta_\Delta \cup \gamma_{\Delta^c})} \Pbstar(d\eta \mid \gamma_{\Delta^c})\leb_\Delta(dx).
\end{multline*}
Now if we consider the measure
\[
	\Pbtstar(d\eta \mid \gamma_{\Delta^c}) = e^{\beta \HDbstar(\eta_\Delta, \gamma_{\Delta^c})} \Pbstar(d\eta \mid \gamma_{\Delta^c}),
\]
a similar computation to the one done in the proof of Proposition 5 in \cite{dereudre2019introduction} gives 
\[
	\int \sum_{x\in\eta_\Delta} f(x,\eta_\Delta \setminus \{x\}) \Pbtstar(d\eta \mid \gamma_{\Delta^c}) = \iint f(x,\gamma) \leb_\Delta(dx) \Pbtstar(d\eta \mid \gamma_{\Delta^c}).
\]
The previous equation shows that $\widetilde{\P}^\beta_{\star \mid \Delta}(\cdot \mid \gamma_{\Lambda^c})$ satisfies the so-called Slivnyak-Mecke formula, which implies that  $\widetilde{\P}^\beta_{\star \mid \Delta}(d\eta \mid \gamma_{\Lambda^c}) = c(\gamma_{\Lambda^c})\Pi_{\Delta}(d\eta)$. The restriction of $\Pbstar(\cdot \mid \gamma_{\Delta^c})$ is proportional to the Poisson point process in $\Delta$ (with a factor depending on $\gamma_{\Delta^c}$). Since we have $\widetilde{\P}^\beta_{\star \mid \Delta}(\emptyset \mid \gamma_{\Lambda^c}) = \Pbstar(N_\Delta = 0 \mid \F_{\Delta^c})(\gamma)$,
we deduce from Theorem~\ref{T:point.deletion} that $c(\gamma_{\Delta^c}) \in (0,+\infty)$ and
\[
	\P^\beta_{\star \mid \Delta}(d\eta \mid \gamma_{\Delta^c}) = c(\gamma_{\Lambda^c}) e^{-\beta \HDbstar(\eta, \gamma_{\Delta^c})}\Pi_{\Delta}(d\eta).
\]
It proves that the normalisation constant introduced in Theorem~\ref{T:grand.canonique} is finite and non-zero since $Z^\beta_{\star,\Delta}(\gamma) = 1/c(\gamma_{\Lambda^c})$. Finally, if $f$ is a measurable function,  from 
\[
	\int f(\gamma) \Pbstar(d\gamma) = \iint f(\eta\cup\gamma) \P^\beta_{\star \mid \Delta}(d\eta \mid \gamma_{\Delta^c}) \Pbstar(d\gamma), 
\]
we obtain the grand canonical DLR equations.
\end{proof}

\section*{Acknowledgement}

This work was supported in part by the Labex CEMPI (ANR-11-LABX-0007-01), the ANR projects PPPP (ANR-16-CE40-0016) and RANDOM (ANR-19-CE24-0014) and by the CNRS GdR 3477 GeoSto.

\bibliographystyle{plain}
\bibliography{biblio}
\end{document}